\documentclass[12pt]{amsart}
\usepackage{amsmath}
\usepackage{amssymb}
\usepackage{amsfonts}
\usepackage{amsthm}
\usepackage{mathrsfs}
\usepackage[all]{xy}
\usepackage[pdftex]{graphicx}
\usepackage{color}
\usepackage{cite}

\begin{document}

\theoremstyle{plain}
\newtheorem{thm}{Theorem}
\newtheorem{lemma}[thm]{Lemma}
\newtheorem{cor}[thm]{Corollary}
\newtheorem{conj}[thm]{Conjecture}
\newtheorem{prop}[thm]{Proposition}
\newtheorem{heur}[thm]{Heuristic}

\theoremstyle{definition}
\newtheorem{defn}[thm]{Definition}
\newtheorem*{ex}{Example}
\newtheorem*{notn}{Notation}

\title[Extensions for Markov Chains]{Comparison Theory for Markov Chains on Different State Spaces and Application to Random Walk on Derangements}
\author{Aaron Smith}

\address{ICERM, Brown University, Providence, RI}
\email{asmith3@math.stanford.edu}
\date{\today}
\maketitle

\section{Abstract}

Let $X_{t}$ and $Y_{t}$ be two Markov chains, on state spaces $\Omega \subset \widehat{\Omega}$. In this paper, we discuss how to prove bounds on the spectrum of $X_{t}$ based on bounds on the spectrum of $Y_{t}$. This generalizes work of Diaconis, Saloff-Coste, Yuen and others on comparison of chains in the case $\Omega = \widehat{\Omega}$. The main tool is the extension of functions from the smaller space to the larger, which allows comparison of the entire spectrum of the two chains. The theory is used to give quick analyses of several chains without symmetry. The main application is to a `random transposition' walk on derangements.

\section{Introduction}
One major tool in the theory of finite state Markov chain has been the comparison technique introduced by Diaconis and Saloff-Coste in the papers \cite{DiSa93} and \cite{DiSa93b}. This theory allows users to analyze the mixing of a Markov chain in terms of the mixing properties of another Markov chain with the same state space, as long as their stationary distributions are not too different. Practically, this may be useful because a chain of interest can be related to a similar but much simpler or more symmetric chain. In many natural examples, however, one expects Markov chains with different state spaces to have similar behaviour. For example, we might expect that removing a small number of vertices at random from a graph would generally have a small impact on the spectral gap of the associated Markov chains. The bounds in \cite{DiSa93} and \cite{DiSa93b} do not apply in this situation. This paper is based on one way to close this gap in the literature, and we demonstrate the usefulness of this approach by deriving new bounds for several natural chains. See \cite{DGJM06} for a useful survey on different ways to apply existing comparison techniques in different contexts. Essentially all of the techniques in that paper apply in the context of different state spaces. \par 
Sections 3 and 4 paper deals with the theory of comparisons for distinct finite state spaces. These bounds are closely related to those found in \cite{DiSa93} and \cite{DiSa93b}. Although the notation is quite different, they are also closely related to ideas found in the papers \cite{DiSa96} and \cite{DiSa98}. Those papers compared random walks on products of groups to some slightly restricted versions, and in particular had the first special examples of comparison of Markov chains with different state spaces. To our knowledge, the only other example is Raymer's thesis \cite{Raym11}. \par 
Sections 5 and 6 apply the bounds obtained in the first part. We begin by looking at random walks on graphs with `some vertices removed' that are analogous to the random walks on graphs with `some edges removed' studied in \cite{DiSa93}. The main application in this paper deals with a random walk on derangements, obtained by comparison with a similar random walk on permutations. This is a simple example of a Markov chain on permutations with restrictions, part of a class first studied for statistical applications in \cite{DGH99}. Closely related chains have been studied with very different spectral methods in \cite{Hanl96} \cite{Blum11b}; the same chain was studied in \cite{Jian12}. \par 
Sections 7-9 of this paper describe some closely related comparison bounds. This begins in section 7 with an analogous bound for discrete-time Markov chains on continuous state spaces, following the work of \cite{Yuen01}. In section 8, we discuss the removal of some technical conditions, such as laziness. Finally, in section 9, we extend the results to another technique, the spectral profile described in \cite{GMT06}. We then use this to sharpen our mixing bound for an earlier example. As proved in \cite{Koz07}, bounds obtained from the spectral profile are almost right for all chains on finite state spaces. Although we don't derive these comparisons explicitly, the discussion in section 9 applies with few changes to many other comparison inequalities based on functional analysis. An excellent survey of such bounds can be found in \cite{Mon07}.

\section{Notation, Background and Statement of Results}

To begin, we consider a $\frac{1}{2}$-lazy, ergodic, irreducable, reversible Markov chain on state space $\widehat{\Omega}$, with transition kernel $K$ and stationary distribution $\mu$ (see section \ref{SecLosingAssumptions} for remarks on obtaining related results under relaxed assumptions) . We will compare this to another  $\frac{1}{2}$-lazy, ergodic, irreducable, reversible Markov chain on state space $\Omega \subset \widehat{\Omega}$, with transition kernel $Q$ and stationary distribution $\nu$. We will begin by comparing Dirichlet forms and Log-Sobolev constants for these two chains. Throughout, we will assume that we have satisfactory information about the chain $K$ on the larger space, and use this to find bounds for the chain $Q$ on the smaller space. \par 
For a general chain on space $X$ with kernel $P$ and stationary distribution $\pi$, and functions $f$ on $X$, we define the following functions:

\begin{align} \label{FunctionalDefs}
V_{\pi}(f) &= \frac{1}{2} \sum_{x,y \in X} \vert f(x) - f(y) \vert^{2} \pi(x) \pi(y) \\
\mathcal{E}_{P}(f,f) &= \frac{1}{2} \sum_{x,y \in X} \vert f(x) - f(y) \vert^{2} P(x,y) \pi(x) \\
L_{\pi}(f) &= \sum_{x \in X} \vert f(x) \vert^{2} \log \left( \frac{f(x)^{2}}{\vert \vert f \vert \vert_{2,\pi}^{2}} \right) \pi(x) \\
\vert \vert f \vert \vert_{2, \pi}^{2} &= \sum_{x \in X} \vert f(x) \vert^{2} \pi(x)
\end{align}

These quantities will be used to describe the spectral gap and log-Sobolev constants of the associated Markov chains. Recall, if $P$ is a reversible, ergodic, irreducible, $\frac{1}{2}$-lazy kernel, it has $\vert X \vert$ real eigenvalues satisfying

\begin{equation*}
1 = \beta_{0}(P) > \beta_{1}(P) \geq \ldots \geq \beta_{\vert X \vert - 1}(P) \geq 0
\end{equation*}

By the variational characterization of eigenvalues, the spectral gap satisfies 

\begin{equation} \label{EqVarCharBeta}
1 - \beta_{1}(P) = \inf_{f \neq 0} \frac{\mathcal{E}_{P}(f,f)}{V_{\pi}(f)}
\end{equation}

As in \cite{DiSa98}, the log-Sobolev constant can similarly be characterized by

\begin{equation} \label{EqVarCharAlpha}
\alpha(P) = \inf_{f \neq 0} \frac{\mathcal{E}_{P}(f,f)}{L_{\pi}(f)}
\end{equation}

Our general approach, when possible, is to use the following theorem (see Theorem 2.2 of \cite{DiSa98}):

\begin{thm}[Mixing Time Bound via Spectral Gap and Log-Sobolev Constant] \label{ThmLogSobolevMixingBound}
For a $\frac{1}{2}$-lazy reversible Markov chain $X_{t}$ started at $X_{0} = x$, and for $t > 1 + \frac{c}{1 - \beta_{1}(P)} + \frac{1}{4 \alpha(P)} \log \log (\frac{1}{\pi(x)})$,
\begin{equation*}
\vert \vert \mathcal{L}(X_{t}) - \pi \vert \vert \leq 2 e^{-c}
\end{equation*}
\end{thm}

When the log-Sobolev constant $\alpha(P)$ is available, this is often better than the usual bound in terms of just the spectral gap (see Theorem 12.3 of \cite{LPW09}), which gives, for $t >  \frac{c}{1 - \beta_{1}(P)} \log(\frac{1}{\pi(x)})$, the bound

\begin{equation} \label{IneqTvSpectralBoundBasic}
\vert \vert \mathcal{L}(X_{t}) - \pi \vert \vert \leq 2 e^{-c}
\end{equation}

We will see shortly that, for many examples, it will be easy to find a very reasonable bound for $\alpha(P)$ after doing the work needed to bound $\beta_{1}(P)$. \par 

It is now time to compare the functionals described in equation \eqref{FunctionalDefs}. For the remainder of this note, $f$ will denote a function on $\Omega$, and $\widehat{f}$ will denote a function on $\widehat{\Omega}$ satisfying $\widehat{f}(x) = f(x)$ for all $x \in \Omega$. We call such a function an \textit{extension} of $f$. We note that the inequalities

\begin{align*}
V_{\nu}(f) &\leq C_{1} V_{\mu}(\widehat{f}) \\
L_{\nu}(f) &\leq C_{2} L_{\mu}(\widehat{f}) \\
\mathcal{E}_{K}(\widehat{f}, \widehat{f}) &\leq C_{3} \mathcal{E}_{Q}(f,f)
\end{align*}

together with the variational characterizations of $\beta_{1}$ and $\alpha$ given in equations \eqref{EqVarCharAlpha} and \eqref{EqVarCharBeta} imply the following bounds on $\beta_{1}(Q)$ and $\alpha(Q)$ in terms of $\beta_{1}(K)$ and $\alpha(K)$:

\begin{align*} 
1 - \beta_{1}(Q) &\geq \frac{1}{C_{1} C_{3}} (1 - \beta_{1}(K)) \\
\alpha(Q) &\geq \frac{1}{C_{2} C_{3}} \alpha(K)
\end{align*}

Finding a good value for $C_{3}$ is difficult and the main object of this paper, but reasonable bounds on $C_{1}$ and $C_{2}$ can be found immediately. The following lemma is useful when $\mu$ and $\nu$ assign similar values to all points in $\Omega$, which is the case for many natural examples. 

\begin{lemma} [Comparison of Variance and Log-Sobolev Constants] \label{LemmaVarLogSobComp}

Let $\widehat{f}$ be any extension of $f$, and let $C = \sup_{y \in \Omega} \frac{\nu(y)}{\mu(y)}$. Then
\begin{align*}
V_{\nu}(f) &\leq C V_{\mu}(\widehat{f}) \\
L_{\nu}(f) &\leq C L_{\mu}(\widehat{f}) \\
\end{align*}
\end{lemma}

\begin{proof}
Define, for $c$ real (respectively real and strictly positive), the following functionals:

\begin{align*}
V_{\pi}(f,c) &= \sum_{x \in X} \vert f(x) - c \vert^{2} \pi(x) \\
L_{\pi}(f,c) &= \sum_{x \in X} \left( \vert f(x) \vert^{2} \log(\vert f(x) \vert)^{2} - \vert f(x) \vert^{2} \log(c) - \vert f(x) \vert^{2} + c \right) \pi(x)
\end{align*}

Recall that $V_{\pi}(f) = \inf_{c \in \mathbb{R}} V_{\pi}(f,c)$, and it is shown in \cite{HoSt87} that $L_{\pi}(f) = \inf_{c \in \mathbb{R}, c> 0} L_{\pi}(f,c)$. Thus, we can write

\begin{align*}
V_{\nu}(f,c) &= \sum_{x \in \Omega} \vert f(x) - c \vert^{2} \nu(x) \\
&= \sum_{x \in \widehat{\Omega}} \vert \widehat{f}(x) - c \vert^{2} \nu(x) \\
&= \sum_{x \in \widehat{\Omega}} \vert \widehat{f}(x) - c \vert^{2} \frac{\nu(x)}{\mu(x)} \mu(x) \\
&\leq C V_{\mu} (\widehat{f},c)
\end{align*}
which implies $V_{\nu}(f) = \inf_{c \in \mathbb{R}} V_{\nu}(f,c) \leq  \inf_{c \in \mathbb{R}} V_{\mu}(\widehat{f},c)  = C V_{\mu}(\widehat{f})$. An analogous calculation shows that $L_{\nu}(f,c) \leq C L_{\mu} (\widehat{f},c)$, which implies $L_{\nu}(f) \leq C L_{\mu}(\widehat{f})$.
\end{proof}

As with the extension theory built up from \cite{DiSa93}, it is possible to get bounds on the entire spectrum of $Q$, rather than just the second-largest eigenvalue. Unlike that case, this will require the extensions to have some structure. In particular, fix a map $M$ from $\mathbb{R}^{\Omega}$ to $\mathbb{R}^{\widehat{\Omega}}$ so that for all $f \in \mathbb{R}^{\Omega}$, $Mf \in \mathbb{R}^{\widehat{\Omega}}$ is an extension of $f$. Assume that we can show

\begin{equation} \label{IneqLinExtComp}
\mathcal{E}_{K}(Mf, Mf) \le C_{3} \mathcal{E}_{Q}(f,f)
\end{equation}

for all $f \in \mathbb{R}^{\Omega}$. Next, consider a Hermitian matrix $P$ with real eigenvalues $\lambda_{1} \geq \ldots \geq \lambda_{n}$, and define for any subspace $W$ the functions

\begin{align*}
L(W) &= \min \{ \frac{ \langle Pf, f \rangle }{\langle f,f \rangle} \, : \, f \in W \} \\
U(W) &= \max \{ \frac{ \langle Pf, f \rangle }{\langle f,f \rangle} \, : \, f \in W \}
\end{align*}

Then recall from e.g. page 185 of \cite{HoJo85} that the eigenvalues of $P$ satisfy

\begin{equation*}
\lambda_{i} = \max \{ L(W) \, : \, dim(W^{\perp}) = i \} = \min \{ U(W) \, : \, dim(W) = i +1 \}
\end{equation*}

This variational characterization, together with inequality \eqref{IneqLinExtComp} and $C_{1} = \sup_{y \in \Omega} \frac{\nu(y)}{\mu(y)}$, gives the bounds

\begin{align*}
1 - \beta_{i}(Q) &\geq \frac{1}{C_{1} C_{3}} (1 - \beta_{i}(K)) \\
\alpha(Q) &\geq \frac{1}{C_{1} C_{3}} \alpha(K)
\end{align*}

The main difficulty will be to bound the Dirichlet forms $\mathcal{E}_{Q}$ and $\mathcal{E}_{K}$. We begin by restricting our attention to the special class of simple random walks on regular graphs, and then write a bound for general finite Markov chains. \par 
Assume that $K$ is a $\frac{1}{2}$-lazy simple random walk on $\widehat{\Omega}$, with associated graph $\widehat{G} = (\widehat{V}, \widehat{E})$. That is, the kernel is given by:
\begin{equation*}
K(x,y) = \left\{
		\begin{array}{lll}
			\frac{1}{2} & \mbox{if } y = x \\
			\frac{1}{2d} & \mbox{if } (x,y) \in \widehat{E} \\
			0 & \mbox{otherwise }
		\end{array}
	\right.
\end{equation*}

Then let $G = (V,E)$ be a subgraph of $\widehat{G}$, where $V$ is obtained from $\widehat{V}$ by removing $m$ vertices, and $E$ is obtained from $\widehat{E}$ by removing all edges in $\widehat{E}$ adjacent to one of the removed edges. Then let $Q$ be a random walk on $G$ described by

\begin{equation*}
Q(x,y) = \left\{
		\begin{array}{lll}
			\frac{1}{2}(2 - \frac{1}{d} \deg(x)) & \mbox{if } y = x \\
			\frac{1}{d} & \mbox{if } (x,y) \in E \\
			0 & \mbox{otherwise }
		\end{array}
	\right.
\end{equation*}

where $\deg(x)$ is the number of vertices in $G$ adjacent to $x$. $Q$ is the \textit{Metropolis-Hastings} walk associated with base walk $K$ and target distribution uniform on $G$ (see \cite{MRRTT53} for an introduction to the Metropolis-Hastings algorithm). \par 
To describe the comparison, it will be necessary first to choose a specific extension $\widehat{f}$ of $f$. For each vertex  $x \in \widehat{G}$, fix some probability measure $P_{x}[y]$ on $G$, requiring $P_{x}[y] = \delta_{x}[y]$ for $x \in G$. This defines a family of extensions by
\begin{equation} \label{EqLinearFamilyExtensions}
\widehat{f}(x) = \sum_{y \in G} P_{x}[y] f(y)
\end{equation}

 Next, for each pair $(x,y) \in \widehat{E}$, fix a joint measure $P_{x,y}[a,b]$ on $G \times G$ satisfying $\sum_{a} P_{x,y}[a,b] = P_{y}[b]$ for all $b \in G$ and $\sum_{b} P_{x,y}[a,b] = P_{x}[a]$ for all $a \in G$. This is a coupling of the distributions $P_{x}, P_{y}$. \par 
Next, for each $a,b \in G$ with $\sum_{x,y \in \widehat{G}} P_{x,y}[a,b] > 0$, it is necessary to define a \textit{flow} in $G$ from $a$ to $b$. To do so, call a sequence of vertices $\gamma = \{ a = v_{0,a,b}, v_{1,a,b}, \ldots, v_{k[\gamma], a,b} = b \}$ a \textit{path} from $a$ to $b$ if $(v_{i,a,b}, v_{i+1,a,b}) \in E$ for all $0 \leq i < k[\gamma]$. Then let $\Gamma_{a,b}$ be the collection of all paths from $a$ to $b$. Call a function $F$ from paths to $[0,1]$ a \textit{flow} if $\sum_{\gamma \in \Gamma_{a,b}} F[\gamma] = 1$ for all $a,b \in \Omega$. We will often write $G_{a,b}$ for the restriction of $F$ to $\Gamma_{a,b}$. Finally, for a path $\gamma \in \Gamma_{a,b}$, we will label its initial and final vertices by $i(\gamma) = a$, $o(\gamma) = b$. \par 
For fixed measures $\{ P_{x} \}_{x \in \widehat{G}}$, couplings $\{ P_{x,y} \}_{(x,y) \in \widehat{E}}$, and flow $F$, we obtain the following bound on Dirichlet forms:

\begin{thm}[Comparison of Dirichlet Forms for Metropolized Simple Random Walk] \label{ThmDirCompSrw}
For flows, distributions, and paths as described above,
\begin{equation*}
\mathcal{E}_{K}(\widehat{f}, \widehat{f}) \leq \frac{n-m}{n} \mathcal{A} \mathcal{E}_{Q}(f,f)
\end{equation*}
where 

\begin{align*}
\mathcal{A} &= \sup_{(q,r) \in E}  ( 1 + 2\sum_{\gamma \ni (q,r)} k[\gamma] F[\gamma] \sum_{y \notin G} P_{y}[o(\gamma)]  \\
&+ \sum_{\gamma \ni (q,r)} F[\gamma] k[\gamma] \sum_{(x,y) \in \widehat{E}, x,y \notin G} P_{x,y}[i(\gamma),o(\gamma)] )
\end{align*}

\end{thm}

For general Markov chains $K$ and $Q$, define a graph $\widehat{G}$ with vertex set $\widehat{\Omega}$ associated with $K$ by creating an edge $(x,y) \in \widehat{E}$ if $K(x,y) > 0$, and a graph $G$ with vertex set $\Omega$ associated with $Q$ by creating an edge $(x,y) \in E$ if $Q(x,y) > 0$. The same setup then gives the following bound:

\begin{thm}[Comparison of Dirichlet Forms for General Chains] \label{ThmDirGenChain}
For flows, distributions and couplings as described above,
\begin{equation*}
\mathcal{E}_{K}(\widehat{f}, \widehat{f}) \leq \mathcal{A} \mathcal{E}_{Q}(f,f)
\end{equation*}
where 
\begin{align*}
\mathcal{A} &= \sup_{Q(q,r) >0}  \frac{1}{Q(q,r) \nu(q)} ( \sum_{\gamma \ni (q,r)} F[\gamma] k[\gamma] K(i(\gamma),o(\gamma)) \mu(i(\gamma)) \\
&+ 2  \sum_{\gamma \ni (q,r)} k[\gamma] F[\gamma] \sum_{y \notin G} P_{y}[o(\gamma)] K(i(\gamma),y) \mu(i(\gamma)) \\
&+ \sum_{\gamma \ni (q,r)} F[\gamma] k[\gamma] \sum_{(x,y) \in \widehat{E}, x,y \notin G} P_{x,y}[i(\gamma),o(\gamma)] K(x,y) \mu(x) )
\end{align*}

\end{thm}

We will now describe some applications of Theorem \ref{ThmDirCompSrw}. The first is analogous to example 2.1 of \cite{DiSa93b}. Let $K$ be the kernel of the $\frac{1}{2}$-lazy simple random walk on the torus $G = \mathbb{Z}_{n}^{2}$ with edges of the form $( (i,j), (i+1,j) )$ and $((i,j),(i,j+1))$. Then let $v_{1}, v_{2}, \ldots, v_{m} \in G$ be any collection of vertices with the property that no two are in the same square $\{(i,i), (i+1,i), (i,i+1), (i+1,i+1) \}$ in $G$. Then let $Q$ be the Metropolis-Hastings walk associated with $G \backslash \{ v_{1}, v_{2}, \ldots, v_{m} \}$. The following is a general bound on the Dirichlet form of $Q$:

\begin{thm} [Comparison for Random Walk on the Torus with Holes] \label{ThmSimpleTorusExample} 
All functions $f$ on $G \backslash \{ v_{1}, v_{2}, \ldots, v_{m} \}$ have extensions $\widehat{f}$ to $G$ so that
\begin{equation*}
\mathcal{E}_{K}(\widehat{f}, \widehat{f}) \leq 6 \left( 1 - \frac{m}{n} \right) \mathcal{E}_{Q}(f,f)
\end{equation*}
\end{thm}

Not all bounds are so useful. For example, if the removed vertices are of the form $\{ (1,1), (2,2), \ldots, (n-1,n-1) \} \cup \{ (1,1 + \frac{n}{2}), (2,2 + \frac{n}{2}), \ldots, (n-1,n + 1 - \frac{n}{2}) \}$, we have:

\begin{thm} [Comparison for Random Walk on the Torus with Bottleneck] \label{ThmBottleneckTorusExample} 
All functions $f$ on $G \backslash \{ (1,1), (1, 1 + \frac{n}{2}), (2,2 + \frac{n}{2}), \ldots, (n-1, n-1) (n-1,n+1 - \frac{n}{2}) \}$ have extensions $\widehat{f}$ to $G$ so that

\begin{equation*}
\mathcal{E}_{K}(\widehat{f}, \widehat{f}) \leq 8 n^{2} \mathcal{E}_{Q}(f,f)
\end{equation*}

\end{thm}

The result is the same upper bound as is given directly by Cheeger's inequality (see Theorem 13.14 of \cite{LPW09}). As discussed immediately after the proof, it seems impossible to do any better by comparison to the standard simple random walk on the torus using Theorem \ref{ThmDirCompSrw}. \par 

The main example in this paper is an application of Theorem \ref{ThmDirCompSrw} to the problem of sampling from derangements. Recall that a permutation $\sigma \in S_{n}$ is called a \textit{derangement} if, for all $i \in [n]$, $\sigma(i) \neq i$. We will compare the well-known `random transposition' walk on $S_{n}$ to its restriction to the derangements $D_{n}$. More precisely, consider the Cayley graph $\widehat{G}$ with vertex set $\widehat{V} = S_{n}$ and edge set $\widehat{E}$ given by $(x,y) \in \widehat{E}$ if and only if $y^{-1} x$ is a transposition. We will compare the $\frac{1}{2}$-lazy transition kernel $K$ on $\widehat{G}$ to its Metropolized version $Q$ on the restriction to derangements $D_{n} \subset S_{n}$. \par 
Although sampling from the set of derangements is not hard (it is easy to sample from $S_{n}$ and rejection-sampling based on this is fairly efficient), the Markov chain is closely related to several more difficult sampling problems. There has been a great deal of interest in the problem of sampling from permutations with restrictions, beginning with the work of Diaconis, Graham and Holmes in \cite{DGH99}. See also the recent work \cite{Blum12}, \cite{Blum11b}, and \cite{Jian12} and the references contained therein for a discussion of other examples. Our main result is:

\begin{thm} [Dirichlet Form Comparison for the Random Transposition Walk on Derangements] \label{ThmDerangementMixing}
Fix $n \geq 10$. All functions $f$ on $ D_{n} $ have extensions $\widehat{f}$ to $S_{n}$ so that

\begin{equation*}
\mathcal{E}_{K}(\widehat{f}, \widehat{f}) \leq 22(e+1) (1 + \epsilon_{n})  \mathcal{E}_{Q}(f,f)
\end{equation*}
where $\vert \epsilon_{n} \vert \leq \frac{13}{n}$. In the other direction, any function $f$ on $D_{n}$ and any extension $\widehat{f}$ of $f$ to $S_{n}$ satisfies

\begin{equation*}
\mathcal{E}_{K}(\widehat{f}, \widehat{f}) \geq \frac{1}{2e}  \mathcal{E}_{Q}(f,f)
\end{equation*}
\end{thm}

We will show that this easily gives the following bound on the mixing time, improving earlier bounds of $O \left( n^{3} \log(n) \right)$ \cite{Jian12}:

\begin{cor} [Mixing Properties of the Random Transposition Walk on Derangements] \label{CorDerangementMixing}
The random walk described above has spectral gap satisfying
\begin{equation*}
1 - \beta_{1}(Q) = \Omega \left( \frac{1}{n} \right)
\end{equation*}
and log-Sobolev constant
\begin{equation*}
\alpha(Q) = \Omega \left( \frac{1}{n \log(n)} \right)
\end{equation*}
\end{cor}
By Theorem \ref{ThmLogSobolevMixingBound}, there exists some constant $a > 0$ and function $f(C)$ such that $\lim_{c \rightarrow \infty} f(C) = 0$ and for $t = C n + a n \log(n)^{2}$, $\vert \vert \mathcal{L}(X_{t}) - \pi \vert \vert_{TV} \leq f(C)$.

In section \ref{SecContChain}, bounds similar to Theorem \ref{ThmDirGenChain} are developed for discrete-time chains on continuous state spaces. The development follows the discrete theory closely, much as W. K. Yuen's development of comparison theory on continuous state spaces in \cite{Yuen01} follows the discrete theory in \cite{DiSa93b}. \par 
Next, in section \ref{SecSpecProf}, we briefly discuss how these extension ideas interact with a recent and powerful way of looking at Dirichlet forms, the spectral profile. The main results from \cite{GMT06} will be introduced. They will then be used to prove the following improvement of Theorem \ref{ThmSimpleTorusExample}:

\begin{thm} [Improved Comparison for Random Walk on the Torus with Holes] \label{ThmProfileTorusExample} 
If $X_{t}$ is a Markov chain as described in Theorem \ref{ThmSimpleTorusExample}, we have for $t = C n^{2}$,
\begin{equation*}
\vert \vert \mathcal{L}(X_{t}) - \nu \vert \vert_{TV} \leq f(C)
\end{equation*}
for some function $f$ independant of $n$ and the particular vertices removed, with $\lim_{C \rightarrow \infty} f(C) = 0$.
\end{thm}\par 

In particular, these random walks have a mixing time that is $O(n^{2})$. This is a substantial improvement on the bound of $O(n^{2} \log(n))$ obtained by a direct application of inequality \eqref{IneqTvSpectralBoundBasic}, and a small improvement on the bound of $O(n^{2} \log ( \log(n)))$ found by a careful application of Theorem \ref{ThmLogSobolevMixingBound}.

\section{Spectral Gap and Log-Sobolev Estimates}
In this section, we prove Theorem \ref{ThmDirGenChain}:

\begin{proof}
Assume without loss of generality that no paths contain repeated edges, and write
\begin{align*}
\mathcal{E}_{K}(\widehat{f}, \widehat{f}) &= \frac{1}{2} \sum_{x,y \in \widehat{\Omega}} \vert \widehat{f}(x) - \widehat{f}(y) \vert^{2} K(x,y) \mu(x) \\
&=\frac{1}{2} \sum_{x,y \in \Omega} \vert f(x) - f(y) \vert^{2} K(x,y) \mu(x) +  \sum_{ x \in \Omega, y \notin \Omega} \vert f(x) - \widehat{f}(y) \vert^{2} K(x,y) \mu(x) \\
&+ \frac{1}{2} \sum_{ x,y \notin \Omega} \vert \widehat{f}(x) - \widehat{f}(y) \vert^{2} K(x,y) \mu(x) \\
&\equiv \frac{1}{2} R_{1} + R_{2} + \frac{1}{2} R_{3}
\end{align*}

The goal is to compare this to $\mathcal{E}_{Q}(f,f) = \frac{1}{2} \sum_{x,y \in \Omega} \vert f(x) - f(y) ) \vert^{2} Q(x,y) \nu(x)$. We begin by looking at $R_{1}$:

\begin{align*}
R_{1} &= \sum_{x,y \in \Omega} \vert f(x) - f(y) \vert^{2} K(x,y) \mu(x) \\
&= \sum_{x,y \in \Omega} \left\vert \sum_{\gamma \in \Gamma_{x,y}} F[\gamma] \sum_{i=0}^{k[\gamma] - 1} (f(v_{x,y,i+1}) - f(v_{x,y,i})) \right\vert^{2} K(x,y) \mu(x) \\
&\leq \sum_{x,y \in \Omega}  \sum_{\gamma \in \Gamma_{x,y}} F[\gamma] \left\vert \sum_{i=0}^{k[\gamma] - 1} (f(v_{x,y,i+1}) - f(v_{x,y,i})) \right\vert^{2}  K(x,y) \mu(x)  \\
&\leq  \sum_{x,y \in \Omega}  \sum_{\gamma \in \Gamma_{x,y}} F[\gamma] k[\gamma] \sum_{i=0}^{k[\gamma] - 1} (f(v_{x,y,i+1}) - f(v_{x,y,i}))^{2}  K(x,y) \mu(x) 
\end{align*}
And so the coefficient $[(f(q) - f(r))^{2}] R_{1}$ of $(f(q) - f(r))^{2}$ in $R_{1}$ is at most
\begin{equation} \label{IneqR1A}
[(f(q) - f(r))^{2}] R_{1} \leq \sum_{\gamma \ni (q,r)} F[\gamma] k[\gamma] K(i(\gamma),o(\gamma)) \mu(i(\gamma))
\end{equation}

The next step is to bound $R_{2}$, which depends on the measures $P_{x}$ and flow $F$, though not on the couplings $P_{x,y}$. Write:

\begin{align*}
R_{2} &= \sum_{x \in \Omega, y \notin \Omega} \vert f(x) - \widehat{f}(y) \vert^{2} K(x,y) \mu(x) \\
&= \sum_{x \in \Omega, y \notin \Omega} \left\vert \sum_{z \in \Omega} P_{y}[z] (f(x) - f(z)) \right\vert^{2} K(x,y) \mu(x) \\
&\leq \sum_{x \in \Omega, y \notin \Omega} \sum_{z \in \Omega} P_{y}[z] (f(x) - f(z))^{2} K(x,y) \mu(x)\\
\end{align*}

where the last inequality is Cauchy-Schwarz. The next step is to write $(f(x) - f(z))^{2}$ in terms of differences which appear in $S$. To do so, note that

\begin{equation} \label{IneqCongestionBound1}
\begin{aligned}  
(f(x) - f(z))^{2} &= \left( \sum_{\gamma \in \Gamma_{x,z}} F[\gamma] \sum_{i=0}^{k[\gamma] - 1} ( f(v_{x,z,i+1}) - f(v_{x,z,i}) ) \right)^{2} \\
&\leq \sum_{\gamma \in \Gamma_{x,z}} F[\gamma] \left( \sum_{i=0}^{k[\gamma] - 1} ( f(v_{x,z,i+1}) - f(v_{x,z,i}) ) \right)^{2} \\
&\leq \sum_{\gamma \in \Gamma_{x,z}} F[\gamma] k[\gamma] \sum_{i=0}^{k[\gamma] - 1} ( f(v_{x,z,i+1}) - f(v_{x,z,i}) )^{2}
\end{aligned}
\end{equation}
where both inequalities are Cauchy-Schwarz. From this bound, the coefficient of $(f(q) - f(r))^{2}$ in $R_{2}$ is at most

\begin{align} \label{IneqR2A}
[(f(q) - f(r))^{2}]R_{2} \leq \sum_{\gamma \ni (q,r)} k[\gamma] F[\gamma] \sum_{y \notin G} P_{y}[o(\gamma)] K(i(\gamma),y) \mu(i(\gamma))
\end{align}  

Finally, it is necessary to bound $R_{3}$. Write
\begin{align*}
R_{3} &= \sum_{x,y \in \widehat{\Omega} \backslash \Omega} \vert \widehat{f}(x) - \widehat{f}(y) \vert^{2} K(x,y) \mu(x)\\
&= \sum_{ x,y \in \widehat{\Omega} \backslash \Omega} \left\vert \sum_{a \in \Omega} P_{x}[a] f(a) - \sum_{b \in \Omega} P_{y}[b] f(b) \right\vert^{2} K(x,y) \mu(x)\\
&= \sum_{ x,y \in \widehat{\Omega} \backslash \Omega} \left\vert \sum_{a,b \in \Omega} P_{x,y}[a,b] (f(a) - f(b)) \right\vert^{2} K(x,y) \mu(x) \\
&\leq \sum_{ x,y \in \widehat{\Omega} \backslash \Omega} \sum_{a,b \in \Omega} P_{x,y}[a,b] (f(a) - f(b))^{2} K(x,y) \mu(x)
\end{align*}
Using inequality \eqref{IneqCongestionBound1} above, this gives
\begin{equation} \label{IneqR3A}
R_{3} \leq \sum_{ x,y \in \widehat{\Omega} \backslash \Omega} \sum_{a,b \in \Omega} P_{x,y}[a,b] \sum_{\gamma \in \Gamma_{a,b}} F[\gamma] k[\gamma] \sum_{i=0}^{k(\gamma) - 1} ( f(v_{a,b,i+1}) - f(v_{a,b,i}) )^{2} K(x,y) \mu(x)
\end{equation}

In particular, the coefficient of $(f(q) - f(r))^{2}$ in this upper bound is
\begin{align*}
[(f(q) - f(r))^{2}]R_{3} \leq \sum_{\gamma \ni (q,r)} F[\gamma] k[\gamma] \sum_{(x,y) \in \widehat{E}, x,y \notin \Omega} P_{x,y}[i(\gamma),o(\gamma)] K(x,y) \mu(x)
\end{align*}  

Combining inequalities \eqref{IneqR1A}, \eqref{IneqR2A} and \eqref{IneqR3A}, the coefficient of $(f(q) - f(r))^{2}$ in $R_{1} + 2 R_{2} + R_{3}$ is bounded by

\begin{align*}
[(f(q) - f(r))^{2}](R_{1} + 2 R_{2} + R_{3}) &\leq  \sum_{\gamma \ni (q,r)} F[\gamma] k[\gamma] K(i(\gamma),o(\gamma)) \mu(i(\gamma)) \\
&+ 2  \sum_{\gamma \ni (q,r)} k[\gamma] F[\gamma] \sum_{y \notin G} P_{y}[o(\gamma)] K(i(\gamma),y) \mu(i(\gamma)) \\
&+ \sum_{\gamma \ni (q,r)} F[\gamma] k[\gamma] \sum_{(x,y) \in \widehat{E}, x,y \notin G} P_{x,y}[i(\gamma),o(\gamma)] K(x,y) \mu(x) 
\end{align*}

On the other hand, the coefficient of $(f(q) - f(r))^{2}$ in $\mathcal{E}_{Q}(f,f)$ is at least $Q(q,r) \nu(q)$. Thus, setting

\begin{align*}
\mathcal{A} &= \sup_{Q(q,r) >0}  \frac{1}{Q(q,r) \nu(q)} ( \sum_{\gamma \ni (q,r)} F[\gamma] k[\gamma] K(i(\gamma),o(\gamma)) \mu(i(\gamma)) \\
&+ 2  \sum_{\gamma \ni (q,r)} k[\gamma] F[\gamma] \sum_{y \notin G} P_{y}[o(\gamma)] K(i(\gamma),y) \mu(i(\gamma)) \\
&+ \sum_{\gamma \ni (q,r)} F[\gamma] k[\gamma] \sum_{(x,y) \in \widehat{E}, x,y \notin G} P_{x,y}[i(\gamma),o(\gamma)] K(x,y) \mu(x) )
\end{align*}

we have

\begin{equation*}
\mathcal{E}_{K}(\widehat{f}, \widehat{f}) \leq \mathcal{A} \mathcal{E}_{Q}(f,f)
\end{equation*}
which completes the proof.

\end{proof}

Theorem \ref{ThmDirCompSrw} is an immediate corollary.

\section{Simple Examples}

In this section, we present brief proofs of Theorems \ref{ThmSimpleTorusExample} and \ref{ThmBottleneckTorusExample}, which follow quickly from Theorem \ref{ThmDirCompSrw}. We begin with Theorem \ref{ThmSimpleTorusExample}. \par 

\begin{proof}
First, we define the measures. If $x \in \{ v_{1}, v_{2}, \ldots, v_{m} \}$, let

\begin{equation*}
P_{x}[a] = \left\{
		\begin{array}{ll}
			\frac{1}{4} & \mbox{if } (x,a) \in \widehat{E} \\
			0 & \mbox{otherwise }
		\end{array}
	\right.
\end{equation*} 

By assumption, no two vertices in $\{ v_{1}, v_{2}, \ldots, v_{m} \}$ are adjacent, so there are no choices to make when defining the couplings $P_{x,y}$. To define the flow, note that $P_{x,y}[a,b] > 0$ only in three situations. We describe each situation up to swapping coordinates and reflections about rows and columns: \\ \par 
\textbf{Case 1:} $(x,y) = (a,b)$. In this case, $(a, b) \in E$, so define the flow to be concentrated on that single edge. \\ \par 
\textbf{Case 2:} $x = (i,j) \in G$, $y$ of the form $(i+1,j) \notin G$, $a = x$, and $b$ of the form $(i+1,j+1)$. In this case, define the flow to be concentrated on the path $\{ ((i,j), (i,j+1)), ((i,j+1), (i+1,j+1)) \}$. \\ \par 
\textbf{Case 3:} $x = (i,j) \in G$, $y$ of the form $(i+1,j) \notin G$, $a = x$, and $b$ of the form $(i+2,j)$. In this case, there are two length 4 paths between $a$ and $b$, of the form $\{ ((i,j), (i,j+1)), ((i,j+1), (i+1,j+1)), ((i+1,j+1), (i+2,j+1)), ((i+2, j+1), (i+2,j)) \}$ and $\{ ((i,j), (i,j-1)), ((i,j-1), (i+1,j-1)), ((i+1,j-1), (i+2,j-1)), ((i+2, j-1), (i+2,j)) \}$. The flow should put equal weight on both. \par 
Then, note that any edge can be in at most 1 path associated with case 1, 2 paths associated with case 2, and 4 paths associated with case 3. Thus, $\mathcal{A} \leq \frac{n-m}{n}(1 + (2)(2) \frac{1}{4} + (4)(4) \frac{1}{4}) = 6(1 - \frac{m}{n})$.

\end{proof}

As mentioned in the introduction, this bound translates immediately into an $O(n^{2} \log(n))$ bound on the Total Variation mixing time, using inequality \eqref{IneqTvSpectralBoundBasic}. Unfortunately, although Theorem \ref{ThmSimpleTorusExample} can be used to get very good control on the entire spectrum of the associated walk as per the comments immediately preceding Lemma \ref{LemmaVarLogSobComp}, the lack of symmetry in the problem makes it difficult to use the smaller eigenvalues to actually improve our estimate of the mixing time. In section \ref{SecSpecProf}, we will avoid this problem and find the right bound up to the coefficient of the leading term using the spectral profile. \par 

The proof of Theorem \ref{ThmBottleneckTorusExample} is similar:

\begin{proof}
We begin by defining the measures. For $x \in \{ v_{1}, v_{2}, \ldots, v_{m} \}$, define 

\begin{equation*}
P_{x}[a] = \left\{
		\begin{array}{ll}
			\frac{1}{4} & \mbox{if } (x,a) \in \widehat{E} \\
			0 & \mbox{otherwise }
		\end{array}
	\right.
\end{equation*} 

By assumption, no two vertices in $\{ v_{1}, v_{2}, \ldots, v_{m} \}$ are adjacent, so it isn't necessary to define any couplings. To define the flow, put the entire weight on one minimal length path. Since the number of pairs $(a,b)$ with $a$ and $y$ not adjacent but $P_{x,y}[a,b] > 0$ for some $x,y$ is at most $2n$, and the maximal path length is clearly at most $4n$, we can write $\mathcal{A} \leq 8 n^{2}$.
\end{proof}

More importantly, it seems impossible to substantially improve this bound with another comparison to simple random walk on the torus. The missing vertices effectively divide the torus into two regions. There must be at least $\Omega(n)$ paths going between the two regions, and the median path length must also be at least $\Omega(n)$. Since $O(1)$ edges between the two regions exist, any path argument gives $\mathcal{A} = \Omega(n^{2})$. \par 

\section{The Random Transposition Walk on Derangements}

This section contains the proofs of Theorem \ref{ThmDerangementMixing} and Corollary \ref{CorDerangementMixing}. The proof is based on an application of Theorem \ref{ThmDirCompSrw}, and the strategy is quite simple. Say that $\tau$ is an \textit{extension} of $\sigma$ if every cycle of $\sigma$ is contained in some cycle of $\tau$, when cycles are viewed as subsets of $[n]$. Roughly, for $x \in S_{n} \backslash D_{n}$, we will define measures $P_{x}$ supported on $D_{n}$ which are fairly uniform on derangements that are extensions of $x$. We will then find a coupling $P_{x,y}[\sigma,\tau]$ of $P_{x}[\sigma]$ and $P_{y}[\tau]$ so that if $P_{x,y}[\sigma, \tau] > 0$, there will be a sequence of derangements of length at most 4, starting with $\sigma$ and ending with $\tau$, where adjacent derangements differ by a single transposition. Finally, the flows will be supported on these minimal-length paths. To complete the proof, we will describe for any fixed pair $q,r$ of derangements all pairs $x,y$ and all pairs $\sigma,\tau$ so that the edge $(q,r)$ is in a path from $\sigma$ to $\tau$ with $P_{x,y}[\sigma,\tau] > 0$. Most of the details of the proof consist of examining a relatively large number of simple cases. For this reason, we omit some details for cases very similar to those already discussed and instead have only references to the first time the calculation is done. \par 

We begin by setting some notation. Let $S_{n}$ and $D_{n}$ be the collection of permutations and derangements on $[n]$, respectively. For $x \in S_{n}$, define $Fix(x) = \{ i \in [n] \, : \, x[i] = i \}$ to be the fixed points of $x$. We will multiply permutations from left to right, so that e.g. $(1,3)(1,2) = (1,3,2)$. Finally, for $\sigma \in S_{n}$ and any subset $S$ of $n$ which is exactly the union of cycles of $\sigma$, we will denote by $\sigma \vert_{S}$ the restriction of $\sigma$ to $S$. For example, $(125)(34)\vert_{\{1,2,5\} } = (125)$. This is often useful for writing down explicit paths along which most such restrictions don't change. \par 

The next step is to describe the measures $P_{x}$ for $x \in S_{n} \backslash D_{n}$, their couplings $P_{x,y}$ for $x,y \in S_{n} \backslash D_{n}$, and flows $G_{\sigma, \tau}$ for pairs $\sigma, \tau \in D_{n}$ with $P_{x,y}[\sigma, \tau] > 0$ for some pair $x,y \in S_{n}$. We begin with the measures $P_{x}[\sigma]$. Let $Fix(x) = F = (f_{1}, \ldots, f_{m})$, and write the remaining cycles of $x$ as $x = C_{1} C_{2} \ldots C_{k}$, where $C_{i} = (p_{i,1}, p_{i,2}, \ldots, p_{i, \ell(i)})$, with $\ell(i) \geq 2$. \par 

We now write several measures associated with $x \neq id$. Fix $z \in S_{m}$, and construct $\sigma$ distributed according to $P_{x}^{(z)}$ as follows. For $1 \leq i \leq m$, let $a_{i}$ be chosen uniformly in $[n] \backslash \cup_{j=i}^{m} \{ f_{z(j)} \}$. Then write 

\begin{equation} \label{EqDerangementMeasureRep}
\sigma = x (a_{1}, f_{z(1)}) (a_{2}, f_{z(2)}) \ldots (a_{m}, f_{z(m)})
\end{equation}

Note that the $i$'th transposition $(a_{i}, f_{z(i)})$ is the first time that $f_{z(i)}$ appears in the sequence $( x \prod_{i=1}^{j} (a_{i}, f_{z(i)}) )_{j=0}^{m}$, and so the $j$'th term in that sequence is obtained from the $j-1$'st by adding $f_{\kappa(j)}$ to the cycle containing $a_{z(j)}$. In particular, no cycles are split during this iterative construction. This defines a measure $P_{x}^{(z)}$ concentrated on $D_{n}$. It is worth noting that these measures aren't very uniform. For example, if $Fix(x) = \{ a, b \}$ and $P_{x}[\sigma] > 0$, then $(a,b)$ is not in the cycle decomposition of $\sigma$; if $\vert Fix(x) \vert = n-2$, $P_{x}^{(z)}$ is concentrated on $n$-cycles. We're willing to give up some uniformity to gain the following lemma, which is very useful for constructing couplings:

\begin{lemma} [Order Indifference] For $x \in S_{n} \backslash D_{n}$ with $Fix(x) = \{ f_{1}, f_{2}, \ldots, f_{m} \} \neq [n]$, $A \subset D_{n}$, and $z, z' \in S_{m}$, we have $P_{x}^{(z)}[A] = P_{x}^{(z')}[A]$.
\end{lemma}

Observe that, under any ordering, $P_{x}^{(z)}[\sigma] \in \{ 0, \frac{(n-m-1)!}{(n-1)!} \}$, since each obtainable element can be obtained in a unique way. Next, observe that the supports of $P_{x}^{(z)}$ and $P_{x}^{(z')}$ are the same. In cycle notation, they consist of exactly the derangements that can be obtained by slotting the elements of $Fix(x)$ into the non-trivial cycles of $x$. $\square$ \\ \par 

Using this lemma, define $P_{x}$ to be the single measure $P_{x}^{(z)}$ for some (any) ordering $z$. If $x = id$, then let $P_{id}$ be uniform on $n$-cycles. These distributions are biased in the sense discussed immediately before the lemma. However, we'll see that they are `most biased' for permutations with a large number of fixed points, and there aren't enough of those to be significant. \par 

Having described the measures $P_{x}$, it is necessary to find couplings $P_{x,y}$ for $x,y$ adjacent. The goal will be to ensure that if $P_{x,y}[\sigma,\tau] > 0$, then the distance between $\sigma$ and $\tau$ is small as measured by the graph metric on $D_{n}$ induced by the kernel $Q$. Note that, since $x$ and $y$ are only a transposition away, we can assume without loss of generality that $Fix(x) \subset Fix(y)$.  \par 

This observations gives an easy way to construct a coupling, as long as $x,y \neq id$. Under the representation of $P_{x}$ given by equation \eqref{EqDerangementMeasureRep}, we order $Fix(x)$, then order the elements of $Fix(y)$ to put the elements in $Fix(y) \backslash Fix(x)$ at the front, and the remaining elements in the same order as given in $Fix(x)$. Then let $\{ a_{i} \}_{i=1}^{\vert Fix(y) \vert}$ be the random variables used to build $\tau$ from $P_{y}$ in representation \eqref{EqDerangementMeasureRep}. Construct $\sigma$ from $P_{x}$ using the same choices for $a_{i}$ in representation equation \eqref{EqDerangementMeasureRep} for all $i > \vert Fix(y) \vert - \vert Fix(x) \vert$. This defines the coupling for $x,y \neq id$. For $y = id$ and $x = (i,j)$, we can observe that $P_{id} = P_{(i,j)}$, so we choose the obvious coupling $P_{id, (i,j)}[\sigma,\tau] = P_{id}(\sigma) \textbf{1}_{\sigma=\tau}$. \par 

The next step is to define flows between all pairs $\sigma, \tau \in D_{n}$ such that $P_{x,y}[\sigma, \tau] > 0$ for at least some pair $x,y \in S_{n}$. We will often use the shorthand "assign weight $\beta$ to path $\gamma_{\sigma, \tau}$" for "set $G_{\sigma, \tau}[\gamma_{\sigma,\tau}] = \beta$." \\ \par

\textbf{Case 1:} $x, y \in D_{n}$. In this case, $P_{x,y}[\sigma, \tau] = \delta_{x,y}[\sigma, \tau]$. There is an edge between $\sigma$ and $\tau$, and so we assign weight 1 to that length-1 path. \\ \par 

\textbf{Case 2:}  $x \in D_{n}$, $\vert Fix(y) \vert = 1$. Assume without loss of generality that $Fix(y) = i$. Thus, $x = y (i,j)$ for some $j$. In this case, $P_{x,y} [\sigma, \tau] > 0$ if and only if $\sigma = x = y (i,j)$ and $\tau = y (i,k)$ for some $k \neq j$. Thus, we need to create a flow from $y (i,j)$ to $y (i,k)$ for all permutations $y$ with unique fixed point $i$ and all distinct $j,k \neq i$. For parity reasons, it is clear that there are no paths of length 1. Let $C_{j}$ and $C_{k}$ be the cycles containing $j$ and $k$ respectively in $y$. By assumption, these have size at least 2, so write $C_{j} = (j, a_{1}, \ldots, a_{\ell(j)})$ and $C_{k} = (k, b_{1}, \ldots, b_{\ell(k)})$. There are three subcases to consider: \par 
\textbf{Case 2A:} Assume first that $C_{j} \neq C_{k}$. Then $\{ ( y (i,j), y(i,j)(j,k)), (y(i,j)(j,k), y(i,j)(j,k)(i,j) = y(i,k)) \}$ is a length-two path between $y(i,j)$ and $y(i,k)$, and all vertices are clearly in $D_{n}$. Assign weight 1 to this path. \par
\textbf{Case 2B:} $C_{j} = C_{k}$, with $\vert C_{j} \vert > 2$. In this case, write $C = C_{j} = C_{k} = (j, a_{1}, \ldots, a_{\ell(j)}, k, b_{1}, \ldots, b_{\ell(k)})$. If $\ell(j) > 0$, the path described in case 2A remains in $D_{n}$, and so we assign weight 1 to this path. If $\ell(k) > 0$, an analogous path with $j$'s and $k$'s switched and arrows reversed will remain in $D_{n}$. Since $\vert C_{j} \vert > 2$, we have either $\ell(j) > 0$ or $\ell(k) > 0$. Thus, the only remaining case is: \par 
\textbf{Case 2C:} $C_{j} = C_{k} = (j,k)$. In this case, we will use the assumption that $n \geq 5$. Choose $h \in [n] / \{i,j,k\}$, and write $C_{h} = (h, c_{1}, \ldots, c_{\ell(h)})$ for some $\ell(h) \geq 1$. Also write $S = \{ i,j,k,h, c_{1},\ldots, c_{\ell(h)} \}$. We calculate:
\begin{align*}
y(i,j)\vert_{S} &= (i,j,k)(h, c_{1},\ldots,c_{\ell(h)}) \\
y(i,j)(i,h)\vert_{S} &= (i,j,k,h,c_{1},\ldots,c_{\ell(h)}) \\
y(i,j)(i,h)(j,h)\vert_{S} &= (j,k)(h,c_{1},\ldots,c_{\ell(h)}) \\
y(i,j)(i,h)(j,h)(k,h)\vert_{S} &= (i,k,j,h,c_{1},\ldots,c_{\ell(h)}) \\
y(i,j)(i,h)(j,h)(k,h)(i,h)\vert_{S} &= \tau
\end{align*}

The five permutations described above, when restricted to $S^{c}$, are all equal. These five permutations, without restriction to $S$, are all in $D_{n}$, and so for each $h \in [n] / \{i,j,k\}$ this sequence defines a length-4 path from $\sigma$ to $\tau$. In this case, we put weight $\frac{1}{n-3}$ on each of these paths from $\sigma$ to $\tau$. \\ \par 

\textbf{Case 3:} $x \in D_{n}$, $\vert Fix(y) \vert = 2$. Without loss of generality, write $Fix(y) = \{ a, b \}$, so that $x = y (a,b)$. Note also that since $x \in D_{n}$, $P_{x}[\sigma] = \delta_{x}[\sigma]$. Write $y$ in cycle notation as $y = (p_{1,1}, \ldots, p_{1, l(1)})(p_{2,1}, \ldots, p_{2, \ell(2)}) \ldots (p_{k, 1}, \ldots, p_{k, \ell(k)})(a)(b)$, where $\ell(i)$ is the length of the $i$'th longest cycle, with ties broken lexicographically by smallest element. If $P_{y}[\tau] > 0$, we can write $\tau = y (a, p_{i(a), j(a)}) (b, p_{i(b), j(b)})$ or $\tau = y (a, p_{i(a), j(a)})(b,a)$, where in both cases $1 \leq i(a), i(b) \leq k$, $1 \leq j(a) \leq \ell(i(a))$, $1 \leq j(b) \leq \ell(i(b))$. This leads to three types of paths. \par 

\textbf{Case 3A:} $\tau = y (a, p_{i(a), j(a)}) (b, p_{i(b), j(b)})$ with $i(a) \neq i(b)$. We define two paths from $\tau$ to $\sigma$ are as follows. In both paths, the first derangement is $\tau$. The second is given by either $\tau (b, p_{i(a), j(a)})$ or $\tau (a, p_{i(b), j(b)})$. The symmetry between these two first steps being clear, we continue describing only the path beginning $(\tau, \tau (b, p_{i(a), j(a)}), \ldots)$. Note that the cycle structure of $\tau (b, p_{i(a), j(a)})$ is given by the cycle structure of $\tau$ with the two cycles $(p_{i(a), 1}, \ldots, p_{i(a), j(a) - 1}, a, p_{i(a), j(a)}, \ldots, p_{i(a), l(i(a))})$ and \\ $(p_{i(b), 1}, \ldots, p_{i(b), j(b) - 1}, b, p_{i(b), j(b)}, \ldots, p_{i(b), l(i(b))})$ merged into the single cycle 
\begin{equation*}
(p_{i(a),j(a)}, p_{i(a), j(a) + 1}, \ldots, p_{i(a), j(a) - 1}, a, b, p_{i(b), j(b)}, \ldots, p_{i(b), j(b) -1})
\end{equation*}
In particular, it is still a derangement. The next step on this path is $\tau (b, p_{i(a), j(a)}) (a, p_{i(b), j(b)})$. The cycle structure of this permutation is obtained from that of $\tau (b, p_{i(a), j(a)})$ by splitting the large cycle $(p_{i(a),j(a)}, p_{i(a), j(a) + 1}, \ldots, p_{i(a), j(a) - 1}, a, b, p_{i(b), j(b)}, \ldots, p_{i(b), j(b) -1})$ into the smaller cycles $(a,b)$ and $(p_{i(a),j(a)}, p_{i(a), j(a) + 1}, \ldots, p_{i(a), j(a) - 1}, p_{i(b), j(b)}, \ldots, p_{i(b), j(b) -1})$. Again, this is a derangement. The final step is multiplying by $(p_{i(b), j(b)}, p_{i(a), j(a)})$ to get to $\sigma$. We assign weight $\frac{1}{2}$ to both paths. \par 

\textbf{Case 3B:} $\tau = y (a, p_{i(a), j(a)}) (b, p_{i(b), j(b)})$ with $i(a) = i(b)$ and $\tau(a) \neq b$, $\tau(b) \neq a$. We create the following two paths from $\tau$ to $\sigma$. The first vertex is $\tau$. The second vertex is $\tau (a, b)$, which has the same cycle structure as $\tau$, with the long cycle $(p_{i(a), j(a) - 1}, a, p_{i(a), j(a)}, \ldots, p_{i(a), j(b)-1}, b, p_{i(a), j(b)}, \ldots)$ split into the cycles \\ $(a, p_{i(a), j(a)}, \ldots, p_{i(a), j(b)-1})$ and \\ $(b, p_{i(a), j(b)}, \ldots, p_{i(a), j(a)-1})$. By the assumption that $a$ and $b$ were not adjacent in the large cycle, both of the small cycles are of size at least $2$, so this is a derangement. The next vertex should be either $\tau (a,b) (p_{i(a), j(a)}, b)$ or $\tau (a,b) (p_{i(a), j(b)}, a)$. As in case 3A, there is obvious symmetry after relabelling $a$ and $b$, and we will continue the description of the first of these paths. Note that  $\tau (a,b) (p_{i(a), j(a)}, b)$ obtained from $\tau (a,b)$ by merging the cycles $(a, p_{i(a), j(a)}, \ldots, p_{i(a), j(b-1)})$ and $(b, p_{i(a), j(b)}, \ldots, p_{i(a), j(a)-1})$ into the single cycle $(a, b, p_{i(a), j(b)}, \ldots, p_{i(a), j(a)-1},  p_{i(a), j(a)}, \ldots, p_{i(a), j(b-1)})$. This is clearly a derangement. Finally, send $\tau (a,b) (p_{i(a), j(a)}, b)$ to $\sigma$ by multiplying by the transposition $(a, p_{i(a), j(b)})$. The path with the other middle edge is analogous; we assign weight $\frac{1}{2}$ to both paths. \par 

\textbf{Case 3C:} This covers the cases $\sigma = y (a, p_{i(a), j(a)})(b,a)$ and  $\tau = y (a, p_{i(a), j(a)}) (b, p_{i(b), j(b)})$ with $i(a) = i(b)$ and either $\tau(a) = b$ or $\tau(b) = a$. In this case, $\sigma$ is adjacent to $\tau$, and in particular $\sigma = \tau (p_{i(a), j(a)}, b)$. The flow should put all weight on this length-1 path. \\ \par 

The next step is to look at the cases where $x,y \in S_{n} \backslash D_{n}$. There will be 3 cases, depending on whether $\vert Fix(x) \vert = \vert Fix(y)\vert$, $\vert Fix(x) \vert = \vert Fix(y)\vert - 1$, or $\vert Fix(x) \vert = \vert Fix(y)\vert - 2$. These will turn out to be very similar to cases 1 through 3 above, with slightly more complicated notation. In particular, all paths will again be of length at most 4. \\ \par 

\textbf{Case 4:} $\vert Fix(x) \vert = \vert Fix(y)\vert$. In this case, we can assume without loss of generality that $y$ has the same cycle structure as $x$ with the $i$'th cycle, $(p_{i,1}, \ldots, p_{i,\ell(i)})$, split into the two cycles $(p_{i,1}, \ldots, p_{i, a})$ and $(p_{i,a+1}, \ldots p_{i, \ell(i)})$. Let $Fix(x) = Fix(y) \equiv F = \{ f_{1}, \ldots, f_{m} \}$. If $P_{x,y}[\sigma, \tau] > 0$, then $\sigma$ and $\tau$ have the same cycle structure, except that the single cycle including $p_{i,1}, \ldots, p_{i,\ell(i)}$ in that order in $\sigma$ is split into two cycles in $\tau$. One contains $p_{i,1}, \ldots, p_{i,a}$ in that order and the other contains $p_{i,a+1}, \ldots, p_{i, \ell(i)}$ in that order. Both will have some elements of $F$ interspersed between the elements of the form $p_{i,q}$, but these interspersed elements will also always be in the same order in $\sigma$ and $\tau$. \par 
In particular, for some $0 \leq \alpha \leq m$, some $z \in S_{m}$ and $\phi: [\alpha] \rightarrow \ell(i)$ we can write $S = \{ p_{i,1}, \ldots, p_{i, \ell(i)} \} \cup \{ f_{z(1)}, \ldots, f_{z(\alpha)} \}$ and
\begin{align*}
\sigma \vert_{S} &= x \prod_{k=1}^{\alpha} (p_{i, \phi(k)}, f_{z(k)}) \\
\tau \vert_{S} &= x (p_{i, 1}, p_{i, a+1}) \prod_{k=1}^{\alpha} (p_{i, \phi(k)}, f_{z(k)}) \\
\end{align*}
It is easy to check that $\sigma$ and $\tau$ are adjacent, and in fact $\sigma = \tau (\tau[p_{i,a}], \tau[p_{i, \ell(i)}] )$. Assign weight 1 to this length-1 path. \\
 
\textbf{Case 5:} $\vert Fix(x) \vert = \vert Fix(y)\vert - 1$. Assume without loss of generality that $Fix(x) = \{ f_{1}, \ldots, f_{m} \}$ and $Fix(y) = \{i, f_{1}, \ldots, f_{m} \}$. Therefore, $x = y (i,j)$ for some $j \neq i$, and $P_{x,y}[\sigma, \tau] > 0$ if and only if $\sigma, \tau$ can be written in the form
\begin{align} \label{EqSigmaTauRepCase5}
\sigma &= y (i,j) \prod_{s=1}^{m} (a_{s}, f_{s}) \\
\tau &= y (i,k) \prod_{s=1}^{m} (a_{s}, f_{s}) 
\end{align}
with $a_{s} \in [n] \backslash \cup_{t \geq s} \{ f_{t} \}$. We will construct paths very similar to those in Case 2. For $\alpha \in [n]$, let $F(\alpha) = (f_{1}(\alpha), \ldots, f_{\ell'(\alpha)}(\alpha))$ subset of  $F$ with the property that $f_{s} \in F$ if $a_{s} = \alpha$ or $a_{s} \in F$, using the representation \eqref{EqSigmaTauRepCase5}. We order $F(\alpha)$ by the indices, so that if $f_{i}, f_{j} \in F(\alpha)$ with $i < j$, then $f_{i}$ is before $f_{j}$ in the list $F(\alpha)$. When $F(\alpha)$ is empty, define $f_{1}(\alpha) = \alpha$. \par 
Next, say that a directed edge $(\eta, \kappa)$ in $S_{n}$ is \textit{defined by the transposition $(q,r)$} if $\eta^{-1} \kappa = (q,r)$. We will be changing paths by changing the transpositions that define their edges. In general, if $\gamma_{\sigma,\tau} = (\sigma, \sigma (q_{1}, r_{1}), \ldots, \sigma \prod_{i = 1}^{k} (q_{i}, r_{i}) = \tau)$ is a path from $\sigma$ to $\tau$, we say that  $\gamma_{\sigma',\tau'} = (\sigma', \sigma' (q_{1}', r_{1}'), \ldots, \sigma' \prod_{i = 1}^{k} (q_{i}', r_{i}') = \tau')$ is the path from $\sigma'$ to $\tau'$ obtained by replacing all edges defined by $(q_{i}, r_{i})$ to edges defined by $(q_{i}', r_{i}')$. To define the flows in case 5, we will take paths from case 2 and replace all edges defined by transpositions $(q,r)$ with edges defined by transposition $(f_{1}(q), f_{1}(r))$. \par 
We will say this more carefully for the analogue to case 2A. Assume $j,k$ aren't in the same cycle in $y$, and let $S$ be the union of all elements in cycles containing $i,j$ or $k$ in $\sigma$ or $\tau$. Then

\begin{align*}
\sigma \vert_{S} &= (F(i), i, F(j), j, b_{1}, \ldots, b_{\ell(j)}) (F(k), k, c_{1}, \ldots, c_{\ell(k)}) \\
\tau \vert_{S} &= (F(i), i, F(k), k, c_{1}, \ldots, c_{\ell(k)}) (F(j), j, b_{1}, \ldots, b_{\ell(j)}) \\
\end{align*}

And so we write the path $\{ (\sigma, \sigma (f_{1}(j), f_{1}(k))), (\sigma (f_{1}(j), f_{1}(k)), \sigma (f_{1}(j), f_{1}(k)) (f_{1}(j), f_{1}(i))) \}$. This replaces the analogous path $\{ (\sigma, \sigma (j, k)), (\sigma (j,k), \sigma (j,k) (i,j)) \}$ from case 2A. As in case 2B, we will use the same path if $j,k$ are in the same cycle in $y$ and $(i,j,k) \notin \sigma, \tau$. \par 

If $(i,j,k) \in \sigma, \tau$, the same discussion as in case 2C shows that the path which goes through $\sigma$, $\sigma(f_{1}(i),f_{1}(h)), \sigma(f_{1}(i),f_{1}(h))(f_{1}(j),f_{1}(h))$, $\sigma(f_{1}(i),f_{1}(h))(f_{1}(j),f_{1}(h))(f_{1}(k),f_{1}(h))$ and finally $\sigma(f_{1}(i),f_{1}(h))(f_{1}(j),f_{1}(h))(f_{1}(k),f_{1}(h))(f_{1}(i),f_{1}(h)) = \tau$ remains in $D_{n}$. \\ \par

\textbf{Case 6:} $\vert Fix(x) \vert = \vert Fix(y)\vert - 2$. Just as case 5 is very similar to case 2, case 6 is very similar to case 3. Assume $Fix(y) = \{ f_{1}, \ldots, f_{m} \}$ and $Fix(x) = \{c,d, f_{1}, \ldots, f_{m} \}$. In particular, $y = x (c,d)$. Write $y$ in cycle notation, as in case 0c. Then if $P_{x,y}[\sigma, \tau] > 0$, analogously to case 3, we can write the pair $(\sigma, \tau)$ in one of the following three ways: \\
\textbf{Case 6A:}
\begin{align*}
\tau = y \prod_{s=1}^{m} (a_{s}, f_{s}) \\
\sigma = y (c, p_{i(c), j(c)}) (d, p_{i(d), j(d)}) \prod_{s=1}^{m} (a_{s}, f_{s}) \\ 
\end{align*}
with $i(c) \neq i(d)$,\\
\textbf{Case 6B:}
\begin{align*}
\tau = y \prod_{s=1}^{m} (a_{s}, f_{s}) \\
\sigma = y (c, p_{i(c), j(c)}) (d, p_{i(d), j(d)}) \prod_{s=1}^{m} (a_{s}, f_{s}) \\ 
\end{align*}
with $i(c) = i(d)$, or \par
\textbf{Case 6C:}
\begin{align*}
\tau = y \prod_{s=1}^{m} (a_{s}, f_{s}) \\
\sigma = y (c, p_{i(c), j(c)}) (d,c) \prod_{s=1}^{m} (a_{s}, f_{s}) \\ 
\end{align*}
where in each case $a_{s} \in [n] \backslash \cup_{t \geq s} \{ f_{t} \}$. These three possibilities correspond exactly to those in cases 3A, 3B and 3C respectively. Just as in case 5, we define flows by taking the paths in cases 3A, 3B and 3C and substituting an edge defined by transposition $(f_{1}(q), f_{1}(r))$ for any edge defined by transposition $(q,r)$ in case 3. \\ \par 

Having defined the measures, couplings, and flows, we will now bound the comparison constant $\mathcal{A}$. We will do this by bounding separately the edges that appear in paths between $\sigma$ and $\tau$ with $P_{x,y}[\sigma, \tau] > 0$ and $x,y$ satisfying the conditions from cases 1 through 6 above. To be more precise, for $\mathcal{F} \in \{ 1, 2, \ldots, 6 \}$ and $x,y \in S_{n}$, say that $(x,y) \in \mathcal{F}$ if $x^{-1}y$ is a transposition and $x,y$ satisfies the conditions of case $\mathcal{F}$ above. We then write:
\begin{align*}
\mathcal{A} &= \sup_{(q,r) \in E}  ( 1 + 2\sum_{\gamma_{x,z} \ni (q,r)} k[\gamma_{x,z}] G_{x,z}[\gamma_{x,z}] \sum_{y \notin G} P_{y}[z]  \\
&+ \sum_{\gamma_{a,b} \ni (q,r)} G_{a,b}[\gamma_{a,b}] k[\gamma_{a,b}] \sum_{(x,y) \in \widehat{E}, x,y \notin G} P_{x,y}[a,b] ) \\
&\leq 1 + \sup_{(q,r)} \sum_{\mathcal{F}} \sum_{\gamma_{x,z} \ni (q,r), (x,z) \in \mathcal{F}} k[\gamma_{x,z}] G_{x,z}[\gamma_{x,z}] \sum_{y \notin G} P_{y}[z]  \\
&+ \sup_{(q,r)} \sum_{\mathcal{F}} \sum_{\gamma_{a,b} \ni (q,r)} G_{a,b}[\gamma_{a,b}] k[\gamma_{a,b}] \sum_{(x,y) \in \widehat{E}, x,y \notin G, (x,y) \in \mathcal{F}} P_{x,y}[a,b] ) \\
\end{align*}
We will then separately bound the weights associated with each of the 6 cases in this sum. In principle, this part of the argument is the same as the weight-counting at the end of the proof of Theorem 5; it is only the larger number of terms in each case that makes it more complicated. \\ \par 

\textbf{Case 1:} $(q,r) \in \gamma_{\sigma, \tau}$, where $G_{\sigma,\tau}[\gamma_{\sigma,\tau}] >0$ and $P_{x,y}[\sigma, \tau] > 0$ for some $x,y \in D_{n}$. This implies that in fact $(q,r) = (x,y)$, and so the total weight in this case is exactly 1.\\ 

\textbf{Case 2:} $(q,r) \in \gamma_{\sigma, \tau}$, where $G_{\sigma,\tau}[\gamma_{\sigma,\tau}] >0$ and $P_{x,y}[\sigma, \tau] > 0$ for some $x \in D_{n}$, $\vert Fix(y) \vert = 1$. Assume $Fix(y) = i$. Then $x = y (i,j)$ for some $j$, and we can write $\sigma = x = y(i,j)$, $\tau = y(i,k)$. If $x,y$ correspond to case 2A or 2B above, the path between $\sigma$ and $\tau$ passes through the permutations $\{ y(i,j), y (i,j)(j,k), y(i,j)(j,k)(i,j) \}$. Thus, either $(q,r) = (y(i,j), y(i,j)(j,k))$ or $(q,r) = ( y(i,j)(j,k), y(i,j)(j,k)(i,j))$. \par 
First, assume $(q,r) = (y(i,j), y(i,j)(j,k))$. Then $q (i,j)$ and $r (j,k) (i,j)$ have a fixed point at $i$. This means that $q[j] = i$ and $r[k] = i$. In particular, once $q,r$ and $i$ have been fixed, so are $j$ and $k$. Thus, for fixed $q,r$ there are at most $n-1$ choices for $i$, and these choices determine $x,y,\sigma$ and $\tau$. Since $P_{x,y}[\sigma, \tau] \in \{0, \frac{1}{n-1} \}$, this means that the total weight is at most $(n-1) \frac{1}{n-1} = 1$. \par 
If $(q,r) = ( y(i,j)(j,k), y(i,j)(j,k)(i,j))$, a similar computation gives the same conclusion. Thus, the total weight for any given edge $(q,r)$ coming from pairs $x,y$ in case 2A or 2B is at most 2. \par 
If $x,y$ correspond to case 2C, there are four possibilities for the pair $(q,r)$, as described in case 2C. Following the notation in that case, we look at the first possibility, $(q,r) = (y(i,j), y(i,j)(i,h))$. Note that $q[i] = j$ and $r[(q(i,j))^{-1}[j]] = h$. In particular, once $q,r,k$ and $i$ have been fixed, $j = q[i]$ can be computed from them, and this information can be used to compute $h = r[(q(i,j))^{-1}[j]]$. Thus, for fixed $q,r$ there are at most $(n-1)(n-2)$ choices of distinct $i,k$, and these choices determine $x,y,\sigma$ and $\tau$. Since $P_{x,y}[\sigma,\tau] \in \{ 0, \frac{1}{n-1} \}$ and the weight of any particular path in case 2C is $\frac{1}{n-3}$, the total weight assigned to any first edge $(q,r)$ by such a path is at most $\frac{n-2}{n-3}$. A similar analysis with the same conclusion applies to the other 3 edges of the length-4 paths described in case 2C. \par 
We conclude that the total weight assigned to any edge $(q,r)$ by vertices $(x,y)$ covered by case 2 is at most $6 \frac{n-2}{n-3}$. \\ \par 

\textbf{Case 3:} $(q,r) \in \gamma_{\sigma, \tau}$, where $G_{\sigma,\tau}[\gamma_{\sigma,\tau}] >0$ and $P_{x,y}[\sigma, \tau] > 0$ for some $x \in D_{n}$, $\vert Fix(y) \vert = 2$. Write in this case $Fix(y) = \{a,b \}$, so $x = y (a,b)$. There are three possibilities for pairs $(\sigma, \tau)$ with $P_{x,y}[\sigma, \tau] > 0$, corresponding to cases 3A, 3B and 3C. We keep the same cycle notation as in case 3 above and begin by looking at case 3A. In this case, the path from $\tau$ to $\sigma$ has the form 
\begin{align*}
\tau &\rightarrow \tau (b, p_{i(a), j(a)}) \\
&\rightarrow \tau (b, p_{i(a), j(a)}) (a, p_{i(b), j(b)}) \\
&\rightarrow \tau (b, p_{i(a), j(a)}) (a, p_{i(b), j(b)}) (p_{i(b),j(b)}, p_{i(a),j(a)}) = \sigma 
\end{align*}
Note that the pair $(q,r)$ can be any of the 3 edges defined by this path. Look for now at edges of the form $(q,r) = (\tau, \tau (b, p_{i(a), j(a)}))$. In this case, since $\tau =   y (a, p_{i(a), j(a)})(b, p_{i(b), j(b)})$, where $y$ has fixed points at $a$ and $b$, we note that $q(b, p_{i(b), j(b)})(a, p_{i(a), j(a)})$ and $r(b, p_{i(a), j(a)}))(b, p_{i(b), j(b)})(a, p_{i(a), j(a)})$ also have fixed points at $a$ and $b$. In particular, $q[a] = p_{i(a), j(a)}$, $q[b] = p_{i(b), j(b)}$, $r[a] = b$ and $r[b] = p_{i(b), j(b)}$. Thus, once $q,r,a,$ and $b$ have been fixed, they determine $(p_{i(a), j(a)}, p_{i(b),j(b)}, \sigma$, and $\tau$. For fixed $q,r$, there are at most $\frac{n(n+1)}{2}$ choices of $(a,b)$. On the other hand, for a given $x,y$, the probability $P_{x,y}[\sigma,\tau]$ for any pair $\sigma, \tau$ associated with $a,b$ is at most $\frac{1}{(n-1)(n-2)}$. Thus, the total weight assigned to this path is at most $\frac{n(n+1)}{2(n-1)(n-2)}$. Comparing the pair $(q,r)$ to the remaining two edges of the path, the same phenomenon holds: the fact that $y$ has two fixed points means that the choice of two parameters determines the entire path, and so again the weight given to these edges is at most $\frac{n(n+1)}{2(n-1)(n-2)}$. Combining these bounds, we see that this case gives a total weight of at most $3 \frac{n(n+1)}{2(n-1)(n-2)}$. \par 
Looking at the second case, $\tau = y (a, p_{i(a), j(a)})(b, p_{i(b), j(b)})$ with $i(a) = i(b)$, gives the same congestion bound of $3 \frac{n(n+1)}{2(n-1)(n-2)}$ with essentially the same proof. The third case,  $\tau = y (a, p_{i(a), j(a)}) (a,b)$, is essentially the same as case 1. As in that case, we have $\sigma$ and $\tau$ adjacent and again determined by the choice of $(a,b)$. Thus, the total congestion in this case is at most 1 for $n > 6$. \par 
Putting these bounds together, the total weight for any given edge $(q,r)$ coming from pairs $(x,y)$ in this case is at most $1 + 6 \frac{n(n+1)}{2(n-1)(n-2)}$. \\ \par 

\textbf{Case 4:} $(q,r) \in \gamma_{\sigma, \tau}$, where $G_{\sigma,\tau}[\gamma_{\sigma,\tau}] >0$ and $P_{x,y}[\sigma, \tau] > 0$ for some $x,y \notin D_{n}$, $\vert Fix(x) \vert = \vert Fix(y) \vert$. As noted in the coupling description for case 4, this means $(q,r) = (\sigma, \tau)$. We now determine the total weight given to the pair $(\sigma, \tau)$ by all pairs $x,y$ with $ \vert Fix(x) \vert = j$. \par 
First, note that for any particular pair $x,y$ with $\vert Fix(x) \vert = j$, we have $P_{x,y}[\sigma, \tau] \leq \frac{(j-1)!}{(n-1)!}$. Next, note that there are at most $\frac{n!}{j! (n-j)!}$ such pairs $x,y$ for which $P_{x,y}[\sigma, \tau] > 0$; we obtain the bound by noting this is the number of ways to choose the $j$ elements of $Fix(x)$. Summing over the size $j$ of $Fix(x) = Fix(y)$, the total weight assigned to $\sigma, \tau$ is at most

\begin{align*}
W &\equiv \sum_{j=1}^{n-2} \frac{n!}{j! (n-j)!} \frac{(j-1)!}{(n-1)!} \\
&= \sum_{j=1}^{n-2} \frac{n}{j} \frac{1}{(n-j)!} \\
&\leq \sum_{j=1}^{\frac{n}{2}} n \frac{1}{\left( \frac{n}{2} \right)!} + \sum_{j=\frac{n}{2}}^{n-2} \frac{2}{(n-j)!} \\
&\leq \frac{n^{2}}{\left( \frac{n}{2} \right)!} + 2 \sum_{j=2}^{\infty} \frac{1}{j!} \\
&\leq 1 + 2(e-1)
\end{align*}

where the last inequality only applies for $n$ sufficiently large that $n^{2} \leq \frac{n}{2}!$, e.g. $n > 10$ suffices. Thus, the total weight for any given edge $(q,r)$ coming from pairs $(x,y)$ in this case is at most $2e + 1$. \\

\textbf{Case 5:} $(q,r) \in \gamma_{\sigma, \tau}$, where $G_{\sigma,\tau}[\gamma_{\sigma,\tau}] >0$ and $P_{x,y}[\sigma, \tau] > 0$ for some $x, y \notin D_{n}$, $\vert Fix(y) \vert = \vert Fix(x) \vert - 1$. The counting of paths for fixed $\vert Fix(y) \vert$ is as in case 2, and finding the weights of each path and summing is as case 4. More precisely, as in case 4, if $\vert Fix(y) \vert = j$, the weight assigned to $(\sigma, \tau)$ by $P_{x,y}$ is at most $\frac{(j-1)!}{(n-1)!}$. By the same argument as in case 2, the total weight for any edge $(q,r) \in \gamma_{\sigma, \tau}$ is at most $6 \frac{n-2}{n-3}$, and there are at most $\frac{n!}{j! (n-j)!}$ pairs $(x,y)$ for which $P_{x,y}[\sigma, \tau] > 0$. Thus, a calculation analogous to the bound on $W$ in case 4 gives a total weight of at most $6 (2e + 1) \frac{n-2}{n-3}$. Combining these arguments gives a total weight of at most $6 (2e + 1) \frac{n-2}{n-3}$. \\ \par 

\textbf{Case 6:} $(q,r) \in \gamma_{\sigma, \tau}$, where $G_{\sigma,\tau}[\gamma_{\sigma,\tau}] >0$ and $P_{x,y}[\sigma, \tau] > 0$ for some $x,y \notin D_{n}$, $\vert Fix(y) \vert = \vert Fix(x) \vert + 2$.  Combining the arguments of 4 and 3 in the same way that case 5 combined the arguments of cases 4 and 2 gives a total weight of at most $(2e + 1)(1 + 3 \frac{n(n+1)}{(n-1)(n-2)})$.\\
 
Putting together the 6 bounds, and noting that all paths are of length at most $4$, we have $\mathcal{A} \leq 2(e+1) \left(2 + 6 \frac{n-2}{n-3} + 3 \frac{n(n+1)}{(n-1)(n-2)} \right)$. This proves the upper bound in Theorem \ref{ThmDerangementMixing}. \par 

To prove the lower bound, define for $0 \leq \epsilon \leq 1$ the distribution $\pi_{\epsilon}$ on $S_{n}$ by $\pi_{\epsilon}(\sigma) = Z$ for $\sigma \in D_{n}$ and $\pi_{\epsilon}(\sigma) = Z \epsilon$ for $\sigma \in S_{n} \backslash D_{n}$, where $\frac{1}{n!} \leq Z \leq \frac{2e}{n!}$ is a normalizing constant. Then define the kernel $Q_{\epsilon}$ to be the Metropolis kernel associated with base chain $Q$ and stationary distribution $\pi_{\epsilon}$. Let $1 = \beta_{\epsilon, 0} \geq \beta_{\epsilon, 1} \geq \ldots \geq \beta_{\epsilon, n!} \geq 0$ be the spectrum of $Q_{\epsilon}$. Note that $Q_{1} = Q$, and $Q_{0} = K$ when restricted to $D_{n}$. By Cauchy's interlacing theorem, if we denote by $\beta_{1}$ the second-largest eigenvalue of $K$, we have $1 - \beta_{0,1} \leq 1 - \beta_{1}$. \par 
Next, we compare $Q_{\epsilon}$ and $Q_{\epsilon'}$ for $\epsilon > \epsilon' >0$. Since $Q_{\epsilon}[\sigma, \tau] > 0$ if and only if $Q_{\epsilon'}[\sigma, \tau] > 0$, all paths can be made length 1. For this choice of paths, there are three types of edges to look at: those between two elements of $D_{n}$, those between an element of $D_{n}$ and an element of $S_{n} \backslash D_{n}$, and finally those between two elements of $S_{n} \backslash D_{n}$. In all three cases, the congestion and path length is 1. In the first case, $\frac{1}{\pi_{\epsilon}(z) Q_{\epsilon}(z,w)} \pi_{\epsilon'}(z) Q_{\epsilon}(z,w) \leq 2e$. In the second case, $\frac{1}{\pi_{\epsilon}(z) Q_{\epsilon}(z,w)} \pi_{\epsilon'}(z) Q_{\epsilon}(z,w) \leq 2e \frac{\epsilon'}{\epsilon} \leq 2e$. In the third case, $\frac{1}{\pi_{\epsilon}(z) Q_{\epsilon}(z,w)} \pi_{\epsilon'}(z) Q_{\epsilon}(z,w) \leq 2e \left( \frac{\epsilon'}{\epsilon} \right)^{2} \leq 2e$. Thus, $\epsilon_{Q_{\epsilon}}(f,f) \leq 2e \epsilon_{Q_{1}}(f,f)$ for all $\epsilon > 0$. Taking the limit as $\epsilon$ goes to 0,  $\epsilon_{Q_{0}}(f,f) \leq 2e \epsilon_{Q_{1}}(f,f)$. Again, by Cauchy's interlacing theorem, $\epsilon_{K}(f,f) \leq 2e \epsilon_{Q_{1}}(\widehat{f}, \widehat{f})$. This completes the proof of Theorem \ref{ThmDerangementMixing}. $\square$ \\ \par 

To prove Corollary \ref{CorDerangementMixing}, we note that

\begin{align*}
1 - \beta_{1}(Q) &= \min_{f \neq 0} \frac{\mathcal{E}_{Q}(f,f)}{V_{\nu}(f)} \\
&\geq \min_{f \neq 0} 132(1 + e) \frac{\mathcal{E}_{K}(\widehat{f}, \widehat{f})}{V_{\mu}(\widehat{f})} \\
&\geq 264(1 + e) \frac{1}{n}
\end{align*}

Where the first inequality is due to Lemma \ref{LemmaVarLogSobComp}, Theorem \ref{ThmDerangementMixing} and the fact that $\frac{\vert S_{n} \vert}{\vert D_{n} \vert} \leq 3$ for $n \geq 10$, and the second inequality is due to the spectral gap estimate in section 9.2 of \cite{Salo04}. The bound on $\alpha(Q)$ is found the same way, and relies on the bound $\alpha(K) = \Omega(\frac{1}{n \log(n)})$ found in \cite{LeYa98}.

\section{Reversibility and Laziness Assumptions} \label{SecLosingAssumptions}

Theorem \ref{ThmDirGenChain} doesn't make use of either the reversibility of $\frac{1}{2}$-laziness assumptions. However, in order to obtain Total Variation mixing time bounds, we used Theorem \ref{ThmLogSobolevMixingBound}, which uses both assumptions. In this section, we briefly discuss the common techniques for avoiding these assumptions. To use comparison without reversibility, Cheeger's inequality is often used. In particular, section 4 of \cite{DGJM06} applies without modification to the setting of this paper.\par 
Avoiding the laziness assumption requires slightly more work, but is more effective. For non-lazy chains, the term $\beta_{1}(P)$ in Theorem \ref{ThmLogSobolevMixingBound} may be replaced by $\max(\beta_{1}(P), \vert \beta_{\vert X \vert}(P) \vert)$. To estimate $ \vert \beta_{\vert X \vert}(P) \vert$, define the following analogue to the Dirichlet form:

\begin{align*}
\mathcal{F}_{P}(f,f) &= \langle (I+P)f, f \rangle \\
&= \frac{1}{2} \sum_{x,y \in X} \vert f(x) + f(y) \vert^{2} P(x,y) \pi(x)
\end{align*}

By inequality 2.3 of \cite{DiSa93b}, if $\mathcal{F}_{K}(\widehat{f}, \widehat{f}) \leq \mathcal{A} \mathcal{F}_{Q}(f,f)$ and $C = \sup_{y \in \Omega} \frac{\nu(y)}{\mu(y)}$, we have
\begin{align*}
\beta_{1}(Q) &\geq \beta_{\vert \Omega \vert}(Q) \\
& \geq -1 + \frac{C}{\mathcal{A}}(1 + \beta_{\vert \Omega \vert}(K)) \\
& \geq -1 + \frac{C}{\mathcal{A}}(1 + \beta_{\vert \widehat{\Omega} \vert}(K)) 
\end{align*}

As per the comments immediately following Lemma 2, it is possible to obtain bounds on more eigenvalues if there is a more structured relationship between $f, \widehat{f}$. \par 

Define paths, flows, extensions, and couplings as in the proof of Theorem \ref{ThmDirGenChain}, with the added requirement that flows must be concentrated on paths $\gamma$ with $\vert \gamma \vert$ odd. For a given edge $e$ in path $\gamma$, let $t_{e}(\gamma)$ be the number of times that $e$ is traversed in $\gamma$. In Theorem \ref{ThmDirGenChain}, we could assume without loss of generality that this was at most 1; in the present context, we can assume that it is at most 2. Then we have the following comparison result:

\begin{thm}[Comparison of Forms for General Chains] \label{ThmNonLazyGenChain}
For flows, distributions and couplings as described above,
\begin{equation*}
\mathcal{F}_{K}(\widehat{f}, \widehat{f}) \leq \mathcal{A} \mathcal{F}_{Q}(f,f)
\end{equation*}
where 
\begin{align*}
\mathcal{A} &= \sup_{Q(q,r) >0}  \frac{1}{Q(q,r) \nu(q)} ( \sum_{\gamma_{x,y} \ni (q,r)} G_{x,y}(\gamma) t_{(q,r)}(\gamma) k[\gamma] K(x,y) \mu(x) \\
&+ 2  \sum_{\gamma_{x,z} \ni (q,r)} t_{(q,r)}(\gamma) k[\gamma_{x,z}] G_{x,z}[\gamma_{x,z}] \sum_{y \notin G} P_{y}[z] K(x,y) \mu(x) \\
&+ \sum_{\gamma_{a,b} \ni (q,r)} G_{a,b}[\gamma_{a,b}] t_{(q,r)}(\gamma) k[\gamma_{a,b}] \sum_{(x,y) \in \widehat{E}, x,y \notin G} P_{x,y}[a,b] K(x,y) \mu(x) )
\end{align*}

\end{thm}

\begin{proof}

Start by writing
\begin{align*}
\mathcal{F}_{K}(\widehat{f}, \widehat{f}) &= \frac{1}{2} \sum_{x,y \in \widehat{\Omega}} \vert \widehat{f}(x) + \widehat{f}(y) \vert^{2} K(x,y) \mu(x) \\
&=\frac{1}{2} \sum_{x,y \in \Omega} \vert f(x) + f(y) \vert^{2} K(x,y) \mu(x) +  \sum_{ x \in \Omega, y \notin \Omega} \vert f(x) + \widehat{f}(y) \vert^{2} K(x,y) \mu(x) \\
&+ \frac{1}{2} \sum_{ x,y \notin \Omega} \vert \widehat{f}(x) + \widehat{f}(y) \vert^{2} K(x,y) \mu(x) \\
&\equiv \frac{1}{2} R_{1} + R_{2} + \frac{1}{2} R_{3}
\end{align*}

The goal is to compare this to $\mathcal{F}_{Q}(f,f) = \frac{1}{2} \sum_{x,y \in \Omega} \vert f(x) + f(y) ) \vert^{2} Q(x,y) \nu(x)$. We begin by looking at $R_{1}$ (note that the assumption of odd path length occurs on the second line):

\begin{align*}
R_{1} &= \sum_{x,y \in \Omega} \vert f(x) + f(y) \vert^{2} K(x,y) \mu(x) \\
&= \sum_{x,y \in \Omega} \left\vert \sum_{\gamma \in \Gamma_{x,y}} G_{x,y}(\gamma) \sum_{i=0}^{k[\gamma] - 1} (-1)^{i}(f(v_{x,y,i+1}) + f(v_{x,y,i})) \right\vert^{2} K(x,y) \mu(x) \\
&\leq \sum_{x,y \in \Omega}  \sum_{\gamma \in \Gamma_{x,y}} G_{x,y}(\gamma) \left\vert \sum_{i=0}^{k[\gamma] - 1} (-1)^{i}(f(v_{x,y,i+1}) + f(v_{x,y,i})) \right\vert^{2}  K(x,y) \mu(x)  \\
&\leq  \sum_{x,y \in \Omega}  \sum_{\gamma \in \Gamma_{x,y}} G_{x,y}(\gamma) k[\gamma] \sum_{i=0}^{k[\gamma] - 1} (f(v_{x,y,i+1}) + f(v_{x,y,i}))^{2}  K(x,y) \mu(x) 
\end{align*}

And so the coefficient of $(f(q) + f(r))^{2}$ in $R_{1}$ is at most
\begin{equation} \label{EqR1B}
[(f(q) + f(r))^{2}] R_{1} \leq \sum_{\gamma_{x,y} \ni (q,r)} G_{x,y}(\gamma) t_{(q,r)}(\gamma) k[\gamma] K(x,y) \mu(x)
\end{equation}

The next step is to bound $R_{2}$, which depends on the measures $P_{x}$ and flows $G_{x,y}$, though not on the couplings $P_{x,y}$. Write:

\begin{align*}
R_{2} &= \sum_{x \in \Omega, y \notin \Omega} \vert f(x) + \widehat{f}(y) \vert^{2} K(x,y) \mu(x) \\
&= \sum_{x \in \Omega, y \notin \Omega} \left\vert \sum_{z \in \Omega} P_{y}[z] (f(x) + f(z)) \right\vert^{2} K(x,y) \mu(x) \\
&\leq \sum_{x \in \Omega, y \notin \Omega} \sum_{z \in \Omega} P_{y}[z] (f(x) + f(z))^{2} K(x,y) \mu(x)\\
&\leq \sum_{x \in \Omega, y \notin \Omega} \sum_{z \in \Omega} P_{y}[z] \sum_{\gamma \in \Gamma_{x,z}} G_{x,z}(\gamma) k[\gamma] \sum_{i=0}^{k[\gamma] - 1} (f(v_{x,z,i+1}) + f(v_{x,z,i}))^{2}  K(x,y) \mu(x) 
\end{align*}

where the last inequality is Cauchy-Schwarz. The next step is to write $(f(x) + f(z))^{2}$ in terms of differences which appear in $S$. To do so, note that
\begin{equation} \label{IneqCongestionBound}
\begin{aligned}
(f(x) + f(z))^{2} &= \left( \sum_{\gamma \in \Gamma_{x,z}} G_{x,z}(\gamma) \sum_{i=0}^{k(\gamma) - 1} (-1)^{i}( f(v_{x,z,i+1}) + f(v_{x,z,i}) ) \right)^{2} \\
&\leq \sum_{\gamma \in \Gamma_{x,z}} G_{x,z}(\gamma) \left( \sum_{i=0}^{k(\gamma) - 1} (-1)^{i}( f(v_{x,z,i+1}) + f(v_{x,z,i}) ) \right)^{2} \\
&\leq \sum_{\gamma \in \Gamma_{x,z}} G_{x,z}(\gamma) k[\gamma] \sum_{i=0}^{k(\gamma) - 1} ( f(v_{x,z,i+1}) + f(v_{x,z,i}) )^{2}
\end{aligned}
\end{equation}
where both inequalities are just Cauchy-Schwarz. From this bound, the coefficient of $(f(q) + f(r))^{2}$ in $R_{2}$ is at most

\begin{equation} \label{EqR2B}
[(f(q) + f(r))^{2}]R_{2} \leq \sum_{\gamma_{x,z} \ni (q,r)} t_{(q,r)}(\gamma) k[\gamma_{x,z}] G_{x,z}[\gamma_{x,z}] \sum_{y \notin G} P_{y}[z] K(x,y) \mu(x)
\end{equation}  

Finally, it is necessary to bound $R_{3}$. Write
\begin{align*}
R_{3} &= \sum_{ x,y in \widehat{\Omega} \backslash \Omega} \vert \widehat{f}(x) + \widehat{f}(y) \vert^{2} K(x,y) \mu(x)\\
&= \sum_{ x,y in \widehat{\Omega} \backslash \Omega} \left\vert \sum_{a \in \Omega} P_{x}[a] f(a) + \sum_{b \in \Omega} P_{y}[b] f(b) \right\vert^{2} K(x,y) \mu(x)\\
&= \sum_{ x,y in \widehat{\Omega} \backslash \Omega} \left\vert \sum_{a,b \in \Omega} P_{x,y}[a,b] (f(a) + f(b)) \right\vert^{2} K(x,y) \mu(x) \\
&\leq \sum_{ x,y in \widehat{\Omega} \backslash \Omega} \sum_{a,b \in G} P_{x,y}[a,b] (f(a) + f(b))^{2} K(x,y) \mu(x)
\end{align*}
Using inequality \eqref{IneqCongestionBound}, this gives
\begin{equation*} 
R_{3} \leq \sum_{ x,y in \widehat{\Omega} \backslash \Omega} \sum_{a,b \in G} P_{x,y}[a,b] \sum_{\gamma \in \Gamma_{a,b}} G_{a,b}(\gamma) k[\gamma] \sum_{i=0}^{k(\gamma) - 1} ( f(v_{a,b,i+1}) + f(v_{a,b,i}) )^{2} K(x,y) \mu(x)
\end{equation*}

In particular, the coefficient of $(f(q) + f(r))^{2}$ in this upper bound is
\begin{equation} \label{EqR3B}
[(f(q) + f(r))^{2}]R_{3} \leq \sum_{\gamma_{a,b} \ni (q,r)} G_{a,b}[\gamma_{a,b}] t_{(a,b)}(\gamma_{a,b}) k[\gamma_{a,b}] \sum_{(x,y) \in \widehat{E}, x,y \notin G} P_{x,y}[a,b] K(x,y) \mu(x)
\end{equation}  

Combining inequalities \eqref{EqR1B}, \eqref{EqR2B} and \eqref{EqR3B}, the coefficient of $(f(q) + f(r))^{2}$ in $R_{1} + 2 R_{2} + R_{3}$ is bounded by
\begin{align*}
[(f(q) + f(r))^{2}](R_{1} + 2 R_{2} + R_{3}) &\leq \sum_{\gamma_{x,y} \ni (q,r)} G_{x,y}(\gamma) t_{(q,r)}(\gamma) k[\gamma] K(x,y) \mu(x) \\
&+ 2  \sum_{\gamma_{x,z} \ni (q,r)} t_{(q,r)}(\gamma)  k[\gamma_{x,z}] G_{x,z}[\gamma_{x,z}] \sum_{y \notin G} P_{y}[z] K(x,y) \mu(x) \\
&+ \sum_{\gamma_{a,b} \ni (q,r)} G_{a,b}[\gamma_{a,b}] t_{(q,r)}(\gamma)  k[\gamma_{a,b}] \sum_{(x,y) \in \widehat{E}, x,y \notin G} P_{x,y}[a,b] K(x,y) \mu(x)
\end{align*}

On the other hand, the coefficient of $(f(q) + f(r))^{2}$ in $\mathcal{F}_{Q}(f,f)$ is at least $Q(q,r) \nu(q)$. Thus, setting

\begin{align*}
\mathcal{A} &= \sup_{Q(q,r) >0}  \frac{1}{Q(q,r) \nu(q)} ( \sum_{\gamma_{x,y} \ni (q,r)} G_{x,y}(\gamma) t_{(q,r)}(\gamma) k[\gamma] K(x,y) \mu(x) \\
&+ 2  \sum_{\gamma_{x,z} \ni (q,r)} t_{(q,r)}(\gamma) k[\gamma_{x,z}] G_{x,z}[\gamma_{x,z}] \sum_{y \notin G} P_{y}[z] K(x,y) \mu(x) \\
&+ \sum_{\gamma_{a,b} \ni (q,r)} G_{a,b}[\gamma_{a,b}] t_{(q,r)}(\gamma) k[\gamma_{a,b}] \sum_{(x,y) \in \widehat{E}, x,y \notin G} P_{x,y}[a,b] K(x,y) \mu(x) )
\end{align*}

this implies

\begin{equation*}
\mathcal{F}_{K}(\widehat{f}, \widehat{f}) \leq \mathcal{A} \mathcal{F}_{Q}(f,f)
\end{equation*}
which completes the proof.

\end{proof}

\section{Comparison for Chains on Continuous State Spaces} \label{SecContChain}

In this section, we write down an analogue to Theorem \ref{ThmDirGenChain} for Markov chains on continuous state spaces, based on Theorem 3.2 of \cite{Yuen01}. It is necessary to develop some notation and definitions. \par 

We consider two state spaces $S \subset \widehat{S}$, with measurable sets $\widehat{F}$ and $F = \{ A \cap S \vert A \in \widehat{F} \}$. Then let $K(x, dy)$ and $Q(x, dy)$ be measurable kernels on $(\widehat{S}, \widehat{F})$ and $(S,F)$ with stationary distributions $\mu$ and $\nu$. Again, the goal will be to describe the mixing properties of $Q$ in terms of the mixing properties of $K$, using spectral information. Although this setup is quite general, and much of the work goes through in greater generality, we will assume that $S$ and $\widehat{S}$ are Lebesgue-measurable subsets of $\mathbb{R}^{n}$. We will write $dx$ for a reference measure on $\widehat{S}$, and we will also assume that $S$ has nonzero measure under $dx$. In particular, if $S$ is a submanifold of a manifold $\widehat{S}$, we allow $\widehat{S}$ to have positive codimension in $\mathbb{R}^{n}$, but don't allow $S$ to have positive codimension in $\widehat{S}$. This zero-codimension assumption cannot be dropped easily; the Markov kernel on $\widehat{S}$ will generally give no information about kernels on subsets of measure 0.  \par 
Say that a Kernel $P$ with stationary distribution $\pi$ is reversible if $\pi (dx) P(x, dy) = \pi (dy) P(y, dx)$. Note that if $P$ is reversible, it is a self-adjoint operator on $L^{2}(\pi)$, and so in particular has a real spectrum. Let $\lambda_{0}(P)$  be the infimum of the spectrum of $P$ on the orthogonal complement of $\textbf{1}$, and let $\lambda_{1}(P)$ be the supremum of this spectrum. As in the discrete case,
\begin{equation*}
\lambda_{1}(P) = \inf \left\{ \frac{(f, (I-P)f)_{\pi}}{\vert \vert f \vert \vert_{L^{2}(\pi)}^{2}} : (f,1)_{\pi} = 0, f \neq 0 \right\}
\end{equation*}
Say that a kernel $P$ is $\alpha$-lazy if we can write $P(x,dy) = \alpha \delta_{x}(dy) + (1-\alpha) \mu_{x}(dy)$, where $\delta_{x}$ is the measure concentrated at $x$ and $\mu_{x}(dy)$ is any measure. If $P$ is $\frac{1}{2}$-lazy, then $\lambda_{0}(P) \geq 0$, and so can essentially be ignored. For $\epsilon > 0$, also define $\vert \vert \gamma_{x,y} \vert \vert_{\epsilon} = \sum_{(u,v) \in \gamma_{x,y}} (k_{u}(v) \rho(u))^{- 2 \epsilon}$ Unlike the discrete case, a bound on $\lambda_{1}(P)$ doesn't immediately give a bound on the total variation distance. Instead, we have only
\begin{equation*}
\vert \vert \tau P^{t} - \pi \vert \vert_{2} \leq \vert \vert \tau - \pi \vert \vert_{2} \vert \lambda_{1}(P) \vert^{n}
\end{equation*}
It is now possible to set up the main comparison theorem. Assume that $K$ and $Q$ are $\frac{1}{2}$-lazy, and furthermore that we can write $K(x,dy) = \frac{1}{2} \delta_{x} + k_{x}(y) dy$, $Q(x,dy) = \frac{1}{2} \delta_{x} + q_{x}(y) dy$, $\mu(dy) = \rho(y) dy$ and $\nu(dy) = \tau(y)dy$ for the reference measure $dy$ on $\widehat{S}$. As in the discrete theory, the first step is to define for all functions $f$ on $S$ their extensions $\widehat{f}$ to $\widehat{S}$. Define for all $x \in \widehat{S} \backslash S$ a measure $r_{x}(a) da$ on $S$, and set $\widehat{f}(x) = \int_{S} f(a) r_{x}(a) da$. It is also necessary to define couplings $r_{x,y} (a,b) da db$ of the measures $r_{x}(a) da$ and $r_{y}(b) db$. \par 
Finally, it is necessary to define paths. This is slightly more complicated than the continuous situation. For fixed $x,y \in \widehat{S}$ and kernel $K$, let a sequence $x = v_{0}, v_{1}, \ldots, v_{k} = y$ be called a \textit{path} from $x$ to $y$ if $k_{v_{i}}(v_{i+1}) > 0$ for all $0 \leq i < k$. Say that $(x,y)$ require a path if $r_{ab}(x,y)>0$ for some pair $a,b$ with $\nu(a) k_{a}(b) > 0$, and denote by $\mathcal{P} \subset S^{2}$ the collection of pairs requiring a path. Then for fixed $x,y$, let $\Gamma_{xy}$ be the collection of paths from $x$ to $y$, and let $G: (x,y) \rightarrow \gamma_{xy} \in \Gamma_{xy}$ be a choice of a single element $\gamma_{xy} \in \Gamma_{xy}$ for each pair $(x,y) \in \mathcal{P}$. For the fixed $\gamma_{xy} \in \Gamma_{xy}$, let $\vert \gamma_{xy} \vert$ be the number of elements in $\gamma$, an let $\gamma_{xy}[i]$ be the $i$'th element. Unlike the discrete case, some regularity assumptions are also needed. \par 
Let $V = \{ (x,y,i) : (x,y) \in \mathcal{P}, 1 \leq i \leq \vert \gamma_{xy} \vert \}$. Say that $G$ satisfies the first regularity condition if the map $T(x,y,i) = (G(x,y)[i-1], G(x,y)[i], \vert G(x,y) \vert, i)$ from $V$ to $S^{2} \times \mathbb{N}^{2}$ is injective. Then, for all $b,i \in \mathbb{N}^{2}$ such that $(u,v,b,i) \in T(V)$ for some $u, v \in S^{2}$, let $W_{b,i} = \{ (u,v) : (u,v,b,i) \in T(v) \}$. Assume that $G$ satisfies the first regularity condition, and define the 1 to 1 map $H_{bi}: W_{b,i} \rightarrow S^{2}$ given by $H_{bi}(u,v) = (x,y)$ where $T(x,y,i) = (u,v,b,i)$. Say that $G$ satisfies the second regularity condition if $H_{bi}$ can be extended to a bijection of open sets with continuous partial derivatives a.e. with respect to Lebesgue measure. For the remainder of this paper, we will denote this extension by $H_{bi}$. \par 
Assuming the two regularity conditions hold, define, for all $b,i \in \mathbb{N}^{2}$ such that $(u,v,b,i) \in T(V)$ for some $u, v \in S^{2}$,  $J_{bi}(u,v)$ to be the Jacobian of the change of variables $H_{bi}(u,v) = (x,y)$. Note that these regularity conditions make the continuous comparison theorem substantially harder to use; they mean that a small change in an edge must result in only a small change of the path between its endpoints. Despite this, continuous versions of Theorems \ref{ThmSimpleTorusExample} and \ref{ThmBottleneckTorusExample} are still easy. \par 

\begin{thm} [Comparison for Chains on Continuous State Spaces] \label{ThmExtCont} 
Under the conditions described above, for all $\epsilon \in \mathbb{R}$,
\begin{equation*}
((I-K)\widehat{f}, \widehat{f})_{\mu} \leq  \mathcal{A}_{\epsilon} ((I-Q)f,f)_{\nu}
\end{equation*}
where 

\begin{align*}
\mathcal{A}_{\epsilon} &=  essup_{(u,v) \in E} \{ (q_{x}(y) \tau(u))^{-(1-2\epsilon)} \\
& \, \, \times \sum_{\gamma_{x,y} \ni (u,v)} \vert \vert \gamma_{x,y} \vert \vert_{\epsilon} k_{x}(y) \rho(x) \vert J_{xy}(u,v) \vert \} \\
&+ essup_{(u,v) \in E} \{ (q_{x}(z) \tau(u))^{-(1-2\epsilon)}  \\
& \, \, \, \, \times \sum_{\gamma_{x,z} \ni (u,v)} \vert \vert \gamma_{x,z} \vert \vert_{\epsilon} \left( \int_{\widehat{S} \backslash S} r_{y}(z) k_{x}(y) dy \right) \rho(x) \vert J_{xz}(u,v) \vert \} \\
&+ essup_{(u,v) \in E} \{ (q_{a}(b) \tau(u))^{-(1-2\epsilon)} \sum_{\gamma_{a,b} \ni (u,v)} \vert \vert \gamma_{a,b} \vert \vert_{\epsilon} (\int \int_{(\widehat{S} \backslash S)\times(\widehat{S} \backslash S)} r_{xy}(a,b) \rho(x) k_{x}(y) dx dy) \\
& \, \, \, \, \times \rho(a) \vert J_{ab}(u,v) \vert \}
\end{align*}
\end{thm}

\begin{proof}

Start by writing
\begin{align*}
((I-K)\widehat{f}, \widehat{f})_{\mu} &= \frac{1}{2} \int \int_{\widehat{S} \times \widehat{S}} (\widehat{f}(x) - \widehat{f}(y))^{2} \rho(x) k_{x}(y) dx dy \\
&= \frac{1}{2} \int \int_{S \times S} (f(x) - f(y))^{2} \rho(x) k_{x}(y) dx dy \\
&+ \int \int_{S \times \widehat{S} \backslash S} (f(x) - \widehat{f}(y))^{2} \rho(x) k_{x}(y) dx dy \\
&+ \frac{1}{2} \int \int_{(\widehat{S} \backslash S) \times (\widehat{S} \backslash S)} (\widehat{f}(x) - \widehat{f}(y))^{2} \rho(x) k_{x}(y) dx dy \\
&\equiv \frac{1}{2} R_{1} + R_{2} + \frac{1}{2} R_{3}
\end{align*}

The goal is to compare this to $((I-Q)f,f)_{\nu} = \frac{1}{2} \int \int_{S \times S} (f(x) - f(y))^{2} \tau(x) q_{x}(y) dx dy$. $R_{1}$ is bounded exactly as in Theorem 3.2 of \cite{Yuen00}:

\begin{equation} \label{EqR1C}
\begin{aligned}
\frac{1}{2}R_{1} &= \frac{1}{2} \int \int_{S \times S} (f(x) - f(y))^{2} \rho(x) k_{x}(y) dx dy \\
&\leq essup_{(u,v) \in E} \{ (q_{x}(y) \tau(u))^{-(1-2\epsilon)} \\
& \, \, \times \sum_{\gamma_{x,y} \ni (u,v)} \vert \vert \gamma_{x,y} \vert \vert_{\epsilon} k_{x}(y) \rho(x) \vert J_{xy}(u,v) \vert \} ((I-Q)f,f)_{\nu}
\end{aligned}
\end{equation}
\par 

Next, we bound $R_{2}$ by writing
\begin{align*}
R_{2} &= \int \int_{S \times \widehat{S} \backslash S} (f(x) - \widehat{f}(y))^{2} \rho(x) k_{x}(y) dx dy \\
&= \int \int_{S \times \widehat{S} \backslash S} (f(x) - \int_{z \in S} r_{y}(z) f(z) dz)^{2} \rho(x) k_{x}(y) dx dy \\
&= \int \int_{S \times \widehat{S} \backslash S} ( \int_{z \in S} r_{y}(z) (f(x) - f(z)) dz)^{2} \rho(x) k_{x}(y) dx dy \\
&\leq  \int \int_{S \times \widehat{S} \backslash S} \int_{z \in S}  (f(x) - f(z))^{2} r_{y}(z) dz \rho(x) k_{x}(y) dx dy \\
&= \int \int_{S \times S} (f(x) - f(z))^{2} \rho(x) ( \int_{\widehat{S} \backslash S} r_{y}(z) k_{x}(y) dy) dx dz
\end{align*}

This last term is bounded exactly as in Theorem 3.2 of \cite{Yuen00}, with $k_{x}(z)$ replaced by $( \int_{\widehat{S} \backslash S} r_{y}(z) k_{x}(y) dy)$: 

\begin{equation}\label{EqR2C}
\begin{aligned}
R_{2} &\leq essup_{(u,v) \in E} \{ (q_{x}(z) \tau(u))^{-(1-2\epsilon)}  \\
& \, \, \, \, \times \sum_{\gamma_{x,z} \ni (u,v)} \vert \vert \gamma_{x,z} \vert \vert_{\epsilon} \left( \int_{\widehat{S} \backslash S} r_{y}(z) k_{x}(y) dy \right) \rho(x) \vert J_{xz}(u,v) \vert \} ((I-Q)f,f)_{\nu}
\end{aligned}
\end{equation}

Finally, we bound $R_{3}$ by writing 

\begin{align*}
R_{3} &= \int \int_{(\widehat{S} \backslash S) \times (\widehat{S} \backslash S)} (\widehat{f}(x) - \widehat{f}(y))^{2} \rho(x) k_{x}(y) dx dy \\
&= \int \int_{(\widehat{S} \backslash S) \times (\widehat{S} \backslash S)} \left( \int_{a \in S} f(a) r_{x}(a) da - \int_{b \in S} f(b) r_{y}(b) db \right)^{2} \rho(x) k_{x}(y) dx dy \\ 
&= \int \int_{(\widehat{S} \backslash S) \times (\widehat{S} \backslash S)} \left( \int_{a,b \in S} (f(a) - f(b)) r_{x,y}(a,b) da db \right)^{2} \rho(x) k_{x}(y) dx dy \\ 
&\leq \int \int_{(\widehat{S} \backslash S) \times (\widehat{S} \backslash S)} \int_{a,b \in S} (f(a) - f(b))^{2} r_{x,y}(a,b) da db \rho(x) k_{x}(y) dx dy \\
&= \int \int_{S \times S} (f(x) - f(y))^{2} \left( \int \int_{(\widehat{S} \backslash S)\times(\widehat{S} \backslash S)} r_{xy}(a,b) \rho(x) k_{x}(y) dx dy \right) da db
\end{align*}

This last term is bounded exactly as in Theorem 3.2 of \cite{Yuen00}, with $\rho(a) k_{a}(b)$ replaced by $(\int \int_{(\widehat{S} \backslash S)\times(\widehat{S} \backslash S)} r_{xy}(a,b) \rho(x) k_{x}(y) dx dy)$:

\begin{equation} \label{EqR3C}
\begin{aligned}
R_{3} &\leq essup_{(u,v) \in E} \{ (q_{a}(b) \tau(u))^{-(1-2\epsilon)} \sum_{\gamma_{a,b} \ni (u,v)} \vert \vert \gamma_{a,b} \vert \vert_{\epsilon} (\int \int_{(\widehat{S} \backslash S)\times(\widehat{S} \backslash S)} r_{xy}(a,b) \rho(x) k_{x}(y) dx dy) \\
& \, \, \, \, \times \rho(a) \vert J_{ab}(u,v) \vert \} ((I-Q)f,f)_{\nu}
\end{aligned}
\end{equation}

The theorem follows by combining inequalities \eqref{EqR1C}, \eqref{EqR2C} and \eqref{EqR3C}.
\end{proof}

\section{Applications to the Spectral Profile} \label{SecSpecProf}

In this section, we prove Theorem \ref{ThmProfileTorusExample} using the techniques found in \cite{GMT06}. First, we recall the notation in that paper. For $S \subset \Omega$, let $c_{0}(S) = \{ f \, : \, f \geq 0, \sup (f) \subset S, f \neq const \}$. Then define, for kernel $Q$ with stationary distribution $\nu$, 

\begin{equation*}
\lambda(S) = \inf_{f \in c_{0}(S)} \frac{\mathcal{E}_{Q}(f,f)}{V_{\nu}(f)}
\end{equation*}

And let $\nu_{\min} = \min_{\omega \in \Omega} \nu(\omega)$. Then define the \textit{spectral profile} associated with $Q$ by:

\begin{equation*}
\Lambda(r) = \inf_{\nu_{\min} \leq \nu(S) \leq r} \lambda(S)
\end{equation*}

Define the spectral profile $\widehat{\Lambda}$ associated with $K$ analogously. The main use of this definition in this context is through the following immediate consequence of Corollary 2.1 of \cite{GMT06}:

\begin{thm}[Spectral Profile Bound] \label{MainSpectralBound} Fix $\epsilon > 0$ and let $X_{t}$ be a $\frac{1}{2}$-lazy, reversible chain with associated spectral profile $\Lambda$. Then for $t > \int_{4 \nu_{\min}}^{4 \epsilon^{-1}} \frac{2}{r \Lambda(r)} dr$,
\begin{equation*}
\vert \vert \mathcal{L}(X_{t}) - \nu \vert \vert_{TV} \leq \epsilon
\end{equation*}
\end{thm}

We will use this bound with the following lemma:

\begin{lemma} [Comparison for Spectral Profile] \label{LemmaSpecProfComp}

Let $M$ be a matrix with nonnegative entries such that $Mf \in \mathbb{R}^{\widehat{\Omega}}$ is an extension of $f$ for all $f \in \mathbb{R}^{\Omega}$. Assume that
\begin{equation*}
\mathcal{E}_{Q}(f,f) \geq \mathcal{A} \mathcal{E}_{K}(Mf, Mf)
\end{equation*}
Furthermore, set $C_{1} = \sup_{x \in \Omega} \frac{\nu(x)}{\mu(x)}$. Finally, for $S \subset \Omega$, define $\widehat{S} \subset \widehat{\Omega}$ to be the support of $M \textbf{1}_{S}$ and $C_{2} = \sup_{S \subset \Omega} \frac{\nu(S)}{\mu(\widehat{S})}$. Then

\begin{equation*}
\Lambda(r) \geq \frac{\mathcal{A}}{C_{1}} \widehat{\Lambda}(C_{2} r)
\end{equation*}
\end{lemma}

\begin{proof}

For all $f \in c_{0}(S)$, 

\begin{align*}
\lambda(S) &\geq \frac{\mathcal{E}_{Q}(f,f)}{V_{\nu}(f)} \\
& \geq \frac{ \mathcal{A} \mathcal{E}_{K}(Mf, Mf)}{C_{1} V_{\mu}(Mf)}
\end{align*}

But by assumption, the support of $Mf$ is contained in $\widehat{S}$. Thus,

\begin{align*}
\lambda(S) \geq \frac{\mathcal{A}}{C_{1}} \widehat{\lambda}(\widehat{S})
\end{align*}

The result follows immediately.
\end{proof}

We will now use this lemma along with Theorems \ref{MainSpectralBound} and \ref{ThmSimpleTorusExample} to prove Theorem \ref{ThmProfileTorusExample}. The distributions and flow will be as in the proof of Theorem \ref{ThmSimpleTorusExample}; it is easy to check that, in the notation of Lemma 14, $C_{2} \leq 4$ in this example. Thus, the only missing ingredient is a bound on the spectral profile $\widehat{\Lambda}$ associated with simple random walk on the torus. By remark 6 following Theorem 1.2 of \cite{DiSa94}, the simple random walk on the torus has a property known as $(\frac{1}{2}, 2)$ moderate growth (see that paper for a definition of \textit{moderate growth}). By Lemma 5.3 of \cite{DiSa96b}, this walk satisfies what is known as a local Poincare inequality, with constant $8$ (again, see that paper for a definition of \textit{local Poincare inequality}). We don't use these two properties directly, but combining them with the inequality following equation 4.3 of \cite{GMT06}, the spectral profile of the random walk on the torus satisfies the inequality
\begin{equation*}
\widehat{\Lambda}(r) \geq \left( \frac{8}{27 r} - 1 \right) \frac{1}{2n^{2}}
\end{equation*}
Thus, by Lemma \ref{LemmaSpecProfComp} and the comments immediately following it, along with Theorem \ref{ThmSimpleTorusExample}, the spectral profile of the walk on the torus with holes satisfies
\begin{equation*}
\Lambda(r) \geq \frac{9}{4 n^{2}} \left( \frac{2}{27 r} - 1 \right)
\end{equation*}
Theorem \ref{ThmProfileTorusExample} follows immediately from this bound and Theorem \ref{MainSpectralBound}.

\bibliographystyle{plain}
\bibliography{DerangementsBib}

\begin{thebibliography}{10}

\bibitem{Blum11b}
Olena Blumberg.
\newblock Permutations with interval restrictions.
\newblock {\em PhD Thesis, Stanford University}, 2011.

\bibitem{Blum12}
Olena Blumberg.
\newblock Cutoff for the transposition walk on permutations with one-sided
  restrictions.
\newblock {\em Preprint}, 2012.

\bibitem{DGH99}
Persi Diaconis, Ronald Graham, and Susan Holmes.
\newblock Statistical problems involving permutations with restricted
  positions.
\newblock {\em IMS Lecture Notes Monogr. Ser}, pages 195--222, 1999.

\bibitem{DiSa93}
Persi Diaconis and Laurent Saloff-Coste.
\newblock Comparison techniques for random walks on finite groups.
\newblock {\em Annals of Probability}, 21(4):2131--2156, 1993.

\bibitem{DiSa93b}
Persi Diaconis and Laurent Saloff-Coste.
\newblock Comparison theorems for reversible markov chains.
\newblock {\em Annals of Applied Probability}, 3(3):696--730, 1993.

\bibitem{DiSa94}
Persi Diaconis and Laurent Saloff-Coste.
\newblock Moderate growth and random walk on finite groups.
\newblock {\em Geom. Func. Anal.}, 4(1):1--36, 1994.

\bibitem{DiSa96b}
Persi Diaconis and Laurent Saloff-Coste.
\newblock Nash inequaltites for finite markov chains.
\newblock {\em Journal of Theoretical Probability}, 9:459--510, 1996.

\bibitem{DiSa96}
Persi Diaconis and Laurent Saloff-Coste.
\newblock Walks on generating sets of abelian groups.
\newblock {\em Probability Theory and Related Fields}, 105:393--421, 1996.

\bibitem{DiSa98}
Persi Diaconis and Laurent Saloff-Coste.
\newblock Walks on generating sets of groups.
\newblock {\em Inventiones Math}, 134(2):250--301, 1998.

\bibitem{DGJM06}
Martin Dyer, Leslie Goldberg, Mark Jerrum, and Russell Martin.
\newblock Markov chain comparison.
\newblock {\em Probability Surveys}, 3:89--111, 2006.

\bibitem{GMT06}
Sharad Goel, Ravi Montenegro, and Prasad Tetali.
\newblock Mixing time bounds via the spectral profile.
\newblock {\em Electronic Journal of Probability}, 2006.

\bibitem{Hanl96}
Phil Hanlon.
\newblock A random walk on the rook placements on a ferrers board.
\newblock {\em Electron. J. Combin.}, 3(2):24, 1996.

\bibitem{HoSt87}
Richard Holley and Dan Stroock.
\newblock Logarithmic sobolev inequalities and stochastic ising models.
\newblock {\em J. Stat. Phys.}, 46:1159--1194, 1987.

\bibitem{HoJo85}
Roger Horn and Charles Johnson.
\newblock {\em Matrix Analysis}.
\newblock Cambridge Univ. Press, 1985.

\bibitem{Jian12}
John Jiang.
\newblock Mixing time of markov chains on finite and compact lie groups.
\newblock {\em PhD Thesis, Stanford University}, 2012.

\bibitem{Koz07}
Gady Kozma.
\newblock On the precision of the spectral profile.
\newblock {\em Alea}, 2007.

\bibitem{LeYa98}
Tzong-Yow Lee and Horng-Tzer Yau.
\newblock Logarithmic sobolev inequality for some models of random walks.
\newblock {\em Annals of Probability}, 26(4):1855--1873, 1998.

\bibitem{LPW09}
David Levin, Yuval Peres, and Elizabeth Wilmer.
\newblock {\em Markov Chains and Mixing Times}.
\newblock American Mathematical Society, Providence, Rhode Island, 2009.

\bibitem{MRRTT53}
N.~Metropolis, A.W. Rosenbluth, Rosenbluth, A.H. Teller, and E.~Teller.
\newblock Equations of state calculations by fast computing machines.
\newblock {\em Journal of Chemical Physics}, 21:1087--1092, 1953.

\bibitem{Mon07}
Ravi Montenegro.
\newblock Duality and evolving set bounds on mixing times.
\newblock {\em Alea}, 2007.

\bibitem{Raym11}
Anastasia Raymer.
\newblock Mixing time of the 15 puzzle.
\newblock {\em PhD Thesis, University of California Davis}, 2011.

\bibitem{Salo04}
Laurent Saloff-Coste.
\newblock Random walks on finite groups.
\newblock {\em Probability on Discrete Structures}, pages 263--346, 2004.

\bibitem{Yuen00}
Wai~Kong Yuen.
\newblock Applications of geometric bounds to the convergence rate of markov
  chains on rn.
\newblock {\em Stoch. Proc. Appl.}, 87:1--23, 2000.

\bibitem{Yuen01}
Wai~Kong Yuen.
\newblock Applications of geometric bounds to convergence rates of of {M}arkov
  chains and {M}arkov processes on rn.
\newblock {\em PhD Thesis, University of Toronto}, 2001.

\end{thebibliography}

\end{document}